\documentclass[a4paper,11pt]{article}
 \usepackage[plainpages=false]{hyperref}
 \usepackage{amsfonts,amsmath,amssymb,amsthm}
 \usepackage{latexsym,lscape,rawfonts,mathrsfs}

 \usepackage[dvips]{color}
 \usepackage{multicol}

 \usepackage[all]{xy}
 \usepackage{eufrak}
 \usepackage{makeidx}
 \usepackage{graphicx,psfrag}

 \usepackage{array,tabularx}

 \usepackage{setspace}

 \usepackage{appendix}

\usepackage{pdfsync}

\usepackage{txfonts}

 \newcommand{\D}[2]{\ensuremath{ \frac{\partial{#1}}{\partial{#2}}}}
 \newcommand{\E}{\ensuremath{\mathbb{E}}}

 \newcommand{\R}{\ensuremath{\mathbb{R}}}
 \newcommand{\Z}{\ensuremath{\mathbb{Z}}}

 \newcommand{\st}{\ensuremath{\sqrt{-1}}}

 \newcommand{\ddb}{\ensuremath{\partial \bar{\partial}}}

 \newcommand{\ba}{\begin{align*}}
 \newcommand{\ea}{\end{align*}}

 \DeclareMathOperator{\diam}{diam}

 \newcommand{\norm}[2]{{ \ensuremath{\|} #1 \ensuremath{\|}}_{#2}}

 \makeatletter
 \def\ExtendSymbol#1#2#3#4#5{\ext@arrow 0099{\arrowfill@#1#2#3}{#4}{#5}}
 
 \makeatother

 \makeatletter
 \def\ExtendSymbol#1#2#3#4#5{\ext@arrow 0099{\arrowfill@#1#2#3}{#4}{#5}}
 \newcommand\longright[2][]{\ExtendSymbol{-}{-}{\rightarrow}{#1}{#2}}
 \makeatother

 \definecolor{hao}{rgb}{1,0.5,0}
 \definecolor{miao}{cmyk}{0.5,0,0.2,0.2}
 \definecolor{qiao}{gray}{0.96}

 \newtheorem{claim}{Claim}

 \newtheorem{corollary}{Corollary}[section]
 \newtheorem{proposition}{Proposition}[section]
 \newtheorem{lemma}{Lemma}[section]
 
 \newtheorem{theorem}{Theorem}[section]

 \newtheorem{remark}{Remark}[section]

 \newtheorem{theoremin}{Theorem}

 \newtheorem{definitionin}{Definitionin}

 \title{On the structure of almost Einstein manifolds}
 \author{Gang Tian\footnote{Partially supported by NSF grant DMS-0804095.}\;,  Bing Wang\footnote{Supported by NSF grant DMS-1006518 and funds from SCGP.}}

\begin{document}
 \maketitle

 \begin{abstract}

  In this paper, we study the structure of the limit space of a sequence of almost Einstein manifolds, which are generalizations
  of Einstein manifolds.  Roughly speaking, such manifolds are the initial manifolds of some normalized Ricci flows whose
  scalar curvatures are almost constants over space-time  in the $L^1$-sense, Ricci curvatures are bounded from below at the initial time.
  Under the non-collapsed condition, we show that the limit space of a sequence of almost Einstein manifolds has
  most properties which is known for the limit space of Einstein manifolds.
  As applications, we can apply our structure results to study the properties of K\"ahler manifolds.
 \end{abstract}

 \tableofcontents

\section{Introduction}

The regularity theory for non-collapsed Einstein manifolds has attracted many studies in last two decades, e.g.,
\cite{An90}, \cite{BKN}, \cite{Tian90}, \cite{CC1}, \cite{CCT} etc. This theory and its extensions have played a crucial role
in K\"ahle geometry, e.g., in constructing canonical metrics on Fano surfaces (c.f. \cite{Tian90}, \cite{CLW}).

Motivated by the study in K\"ahler geometry, in this paper, we prove new regularity results on the Gromov-Hausdorff limits of Riemannian manifolds
with Ricci curvature bounded from below and which are weakly Einstein in an appropriate sense.

To be precise, we assume that $\left( X_i, x_i, g_i \right)$ is a sequence of non-collapsed Riemannian manifolds of dimension $m$ such that
$Ric \geq -(m-1)$. The well-known Gromov compactness theorem states that by taking a subsequence if necessary, $\left( X_i, x_i, g_i \right)$
converges to a length space $\left( \bar{X}, \bar{x}, \bar{g} \right)$ in the Gromov-Hausdorff topology. {\bf A basic problem in the metric geometry
concerns the regularity of the limit $\left( \bar{X}, \bar{x}, \bar{g} \right)$}.
Note that $\bar g$ is merely a length function in the Gromov compactness theorem.
The fundamental work of Cheeger-Colding~\cite{CC1} shows initial and crucial structure properties for $\left( \bar{X}, \bar{x}, \bar{g} \right)$.
In particular, it follows from~\cite{CC1} that tangent cones exist at every point $y\in \bar{X}$. Using these tangent cones, they gave a
regular-singular decomposition of
$\bar X$. A point $y \in \bar{X}$ is called regular or belongs to the regular part $\mathcal{R}$ if every tangent cone
at $y$ is isometric to the Euclidean space $\left( \R^m, 0, g_{\E} \right)$.
A point $y \in X$ is called singular or belongs to the singular part $\mathcal{S}$ if it is not regular, i.e., at $y$,
there exists some tangent cone $\left( \hat{Y}, \hat{y}, \hat{g} \right)$ which is not isometric to the Euclidean space.
Clearly, we have $\bar{X}=\mathcal{R} \cup \mathcal{S}$. In general, it is unknown if $\mathcal{R}$ is open and even if it is open,
it may not be a manifold and $\bar g$ may not arise from a Riemannian metric in any classical senses. It is expected
in general cases that $\mathcal{R}$ has only locally Lipschitz structures at most.
If $g_i$ has uniformly bounded Ricci curvature, then Cheeger-Colding proved that $\mathcal{R}$
is an open manifold and $\mathcal{S}$ has Hausdorff codimension at least $2$. Moreover, $\bar g$ is a $C^{1,\alpha}$-smooth metric.
Furthermore, if $(X_i, g_i)$ is an Einstein manifold, then the convergence to $\bar X$ restricted to $\mathcal{R}$ is actually in
the $C^{\infty}$-topology and $\bar g$ is a smooth Einstein metric in $\mathcal{R}$ because of the regularity results from the PDE theory.
However, in general, even if the convergence is weak, $\mathcal{R}$ can still possibly be a smooth manifold.
In this paper, we study when the limit can have smooth $\mathcal{R}$ and $\bar g$ is an Einstein metric even if the convergence
$\left( X_i, x_i, g_i \right) \to \left( \bar{X}, \bar{x}, \bar{g} \right)$ is only in the weak topology, say the Gromov-Hausdorff topology.
Our study is analogous to the standard regularity problem in studying weak solutions for PDEs.
In the case of the Einstein equation, because of its invariance under diffeomorphisms, there is not a good notion of weak solutions.
Therefore, we first need to make clear what we mean by Einstein metrics in the weak sense.
Now let us introduce the notion of almost Einstein manifolds we want to study.

\begin{definitionin}
A sequence of closed Riemannian manifolds $\left( X_i^m, x_i, g_i \right)$ is called almost Einstein
if the following conditions are satisfied.
\begin{itemize}
\item $Ric(g_i)+g_i \geq 0$.
\item $x_i \in X_i$, and $|B_{g_i}(x_i, 1)|_{d\mu_{g_i}} \geq \kappa$.
\item The flow $\displaystyle \D{}{t} g\,=\, -Ric + \lambda_i g$ has a solution $g(t)$ with $g(0)= g_i$ on $X_i\times [0,1]$,
where $\lambda_i\in [-1,1]$ is a constant. Moreover, $E_i=\int_0^1\int_{X_i} |R-m\lambda_i|d\mu dt \to 0$.
\end{itemize}

\label{dfnin:almostEinstein}
\end{definitionin}

Note that the non-collapsed condition is included in our definition. This is because the condition $\int_0^1\int_{X_i} |R-m\lambda_i|d\mu_{g_i}dt \to 0$
is not sufficient for proving the following results if collapsing occurs.
However, we will not discuss this further in the current paper.

Clearly, if $\int_0^1\int_{X_i} |R-m\lambda_i|d\mu_{g_i}dt \equiv 0$, then this sequence is exactly a sequence of non-collapsed Einstein
manifolds with bounded Einstein constants.
Such a sequence was extensively studied in the literature.
In fact, the condition $\int_0^1\int_{X_i} |R-m\lambda_i|d\mu_{g_i}dt \to 0$ is crucial in establishing the regularity of $\mathcal{R}$.
It turns out that almost Einstein limits have most known properties
of Einstein limits. Our first theorem is as follows.

\begin{theoremin}[Structure theorem in Riemannian case]
  Suppose $(X_i^m, x_i, g_i)$ is a sequence of almost Einstein manifolds.
  Let $\left( \bar{X}, \bar{x}, \bar{g} \right)$ be a Gromov-Hausdorff limit of $\left( X_i, x_i, g_i \right)$, $\bar{\lambda}$ be the limit of $\lambda_i$.

  Then the limit space $\left( \bar{X}, \bar{x}, \bar{g} \right)$ is a metric space with disjoint decomposition $\bar{X}=\mathcal{R} \cup \mathcal{S}$,
  where $\mathcal{R}$ is the regular part of $\bar{X}$, $\mathcal{S}$ is the singular part of $\bar{X}$. They satisfy the following properties.
  \begin{itemize}
     \item  $\left(\mathcal{R}, \bar{g} \right)$ is a smooth, convex, open Riemannian manifold.
     \item  $\displaystyle Ric(\bar{g}) + \bar{\lambda}\bar{g}=0$.
     \item  If $0<p<1$ and $\rho \geq 1$, then $\int_{\mathcal{R} \cap B(\bar{x}, \rho)} |Rm|^{p} d\mu < C(m,\kappa,p,\rho)$.
     \item  If $y \in \mathcal{S}$, $\left( \hat{X}, \hat{y}, \hat{g} \right)$ is a tangent space of $\bar{X}$ at the point $y$, then
       \begin{align*}
	    d_{GH} \left( \left( B_{\hat{g}}(\hat{y}, 1), \hat{g} \right), \left( B(0,1), g_{\E} \right) \right) > \bar{\epsilon}(m),
       \end{align*}
      where $B(0,1)$ is the standard unit ball in $\R^m$.
    \item $\dim_{\mathcal{H}} \mathcal{S} \leq m-2$.
\end{itemize}
  \label{thmin:goodlimit}
\end{theoremin}

Note that the convexity of $\mathcal{R}$ and the integral bound of $|Rm|$ follow directly from the work of~\cite{NaCold} and~\cite{ChNa} respectively.
We list these results here just for completeness of the known results of Einstein limit.

We observe that if $(M_i, x_i,g_i)$ is a sequence of K\"ahler manifolds, then by a result of the first author and Z. Zhang (c.f.~\cite{TZZ}), the Ricci flow
$\displaystyle \D{}{t} g\,=\, -Ric + \lambda_i g$ has a solution with $g(0)=g_i$ on $M_i\times [0,1]$ so long as
$\lambda_i [\omega_i] + \left(e^{\lambda_i t} - 1 \right) c_1(M_i) > 0 $, where $\omega_i$ denotes the K\"ahler form of $g_i$.
Moreover, if $R - n \lambda_i\ge 0$ and its average tends to zero as $i$ goes to infinity, then one can show that $E_i$ tends to zero.
Thus, the third condition of Definition~\ref{dfnin:almostEinstein} is essentially automatic if $R - n \lambda_i\ge 0$ and its average tends to zero.
This shows that the K\"ahler case is better behaved. A natural question is whether or not the same holds for general Riemannian metrics
with Ricci curvature bounded from below. More precisely, can one solve the above Ricci flow
with initial value $g_0$ in $[0,a]$ such that $a$ depends only on the lower bound of Ricci curvature of $g_0$?

The following theorem strengthens Theorem~\ref{thmin:goodlimit} for K\"ahler manifolds. We say that
a sequence of closed K\"ahler manifolds $\left( M_i^n, x_i, g_i, J_i\right)$ is almost K\"ahler-Einstein
if it is almost Einstein of dimension $m=2n$ and satisfies $F_i=\int_{M_i} |Ric-\lambda_i g_i| d\mu_{g_i}\to 0$.

\begin{theoremin}[Structure theorem in K\"ahler case]
  Suppose $(M_i^n, x_i, g_i, J_i)$ is a sequence of almost K\"ahler Einstein manifolds.
  Let $\left( \bar{M}, \bar{x}, \bar{g} \right)$ be a Gromov-Hausdorff limit of $\left( M_i, x_i, g_i \right)$, $\bar{\lambda}$ be the limit of $\lambda_i$.

  Then the limit space $\left( \bar{M}, \bar{x}, \bar{g} \right)$ is a metric space with the regular-singular disjoint decomposition $\bar{M}=\mathcal{R} \cup \mathcal{S}$.
  They satisfy the following properties.
  \begin{itemize}
    \item  There exists a complex structure $\bar{J}$ on $\mathcal{R}$ such that $\left( \mathcal{R}, \bar{g}, \bar{J} \right)$
      is a smooth, convex, open K\"ahler manifold.
    \item  $\displaystyle Ric(\bar{g}) + \bar{\lambda}\bar{g}=0$.
    \item  If $0<p<2$ and $\rho \geq 1$, then $\int_{\mathcal{R} \cap B(\bar{x}, \rho)} |Rm|^{p} d\mu < C(n,\kappa,p,\rho)$.
    \item  If $y \in \mathcal{S}$, $\left( \hat{M}, \hat{y}, \hat{g} \right)$ is a tangent space of $\bar{M}$ at the point $y$, then
       \begin{align*}
	 d_{GH} \left( \left( B_{\hat{g}}(\hat{y}, 1), \hat{g} \right), \left( B(0,1), g_{\E} \right) \right) > \bar{\epsilon}(2n),
       \end{align*}
     where $B(0,1)$ is the standard unit ball in $\R^{2n}$.
    \item $\dim_{\mathcal{H}} \mathcal{S} \leq 2n-4$.
\end{itemize}
\label{thmin:kgoodlimit}
\end{theoremin}

Our proof of the above theorems is based on the works of~\cite{CC1},~\cite{CCT},~\cite{Pe1} et al.
We need to establish two new technical results.
The first one is a pseudo-locality result (Theorem~\ref{thm:pseudo-locality_A22}) which is similar to Theorem 10.1 and 10.3 of~\cite{Pe1}.
Basically, we need to bound curvature along the Ricci flow whenever the initial metric has its Ricci curvature bounded from below and
the volume ratios of its geodesic balls are sufficiently close to the Euclidean one.
Our proof for this pseudo-locality uses an argument due to Perelman.
The second one is a delicate bound of the Gromov-Hausdorff distance between metrics along the Ricci flow (c.f.~Theorem~\ref{thm:Gap_A22}).
This bound plays a role similar to the gap theorem for Einstein limits and is crucial for us to finish the proof of Theorem~\ref{thmin:goodlimit}
and Theorem~\ref{thmin:kgoodlimit}. \\

The organization of this paper is as follows.  In Section 2, we discuss some standard estimates which will be repeatedly used in the
whole paper. In Section 3, we prove a new pseudo-locality result, i.e., Theorem~\ref{thm:pseudo-locality_A22}.
Using this new pseudo-locality, we prove a gap theorem (Theorem~\ref{thm:Gap_A22}) in Section 4.
Then in section 5, we use pseudo-locality theorem, gap theorem and the fact that scalar curvature is almost constant to show the structure theorems
in both Riemannian and K\"ahler cases. Finally, we construct examples of almost K\"ahler Einstein manifolds and discuss
the applications of our structure theorems to K\"ahler geometry.\\

\noindent {\bf Acknowledgment}
 The second named author is very grateful to professor Xiuxiong Chen and professor Simon Donaldson for their constant support.
 He appreciate SCGP (Simons Center for Geometry and Physics) for offering him the wonderful working condition.
 Part of this work was done when the second named author was visiting BICMR (Beijing International Center of Mathematical Research) during the summer of 2011,
 he would like to thank BICMR for its hospitality.

\section{Elementary estimates}

Before we go to discussion in details, let's fix some notations first.
We assume $X$ to be a closed Riemannian manifold of dimension $m \geq 3$,
$M$ to be a closed K\"ahler manifold of complex dimension $n \geq 2$, real dimension $m=2n \geq 4$.
We denote the volume of standard unit ball in $\R^m$ by $\omega_m$.
We say $A<<B$ for two positive quantities $A$ and $B$ if there is a universal small constant $c=c(m)$ such that $A<cB$.
If not mentioned in particular, the constant $C$ may be different from line to line. \\

In this paper, we often assume $\left\{ (X, g(t)), 0 \leq t \leq 1 \right\}$ satisfies the evolution equation
\begin{align}
  \D{}{t} g = -Ric +\lambda_0 g
  \label{eqn:normalized}
\end{align}
for some constant $\lambda_0$ with $|\lambda_0|\leq 1$.  Note that this flow may not preserve the volume. However, by abuse of notation,
we also call (\ref{eqn:normalized}) as a normalized Ricci flow solution.  Define
\begin{align}
  \tilde{g}(s) \triangleq
  \begin{cases}
    &\left( 1-2\lambda_0 s \right) g\left( \frac{\log (1-2\lambda_0 s)}{-\lambda_0} \right),   \; \textrm{if} \; \lambda_0 \neq 0;\\
    &g(2s), \; \textrm{if} \; \lambda_0=0.
  \end{cases}
\label{eqn:normalization}
\end{align}
Then $\displaystyle \D{}{s} \tilde{g}= -2Ric(\tilde{g})$, which is the (unnormalized) Ricci flow equation.
Clearly, $\tilde{g}(0)=g(0)$. For simplicity of notation, define
$h_{ij} \triangleq R_{ij}-\lambda_0 g_{ij}$, $H \triangleq R-m\lambda_0$. Simple calculation yields
\begin{align}
 \D{}{t} h_{ij} = \frac{1}{2}\Delta h_{ij} + R_{pijq}h_{ij} -h_{ip}h_{pj}, \label{eqn:normric}
\end{align}
which implies
\begin{align}
  \D{}{t} |h| \leq \frac{1}{2} \Delta |h| + |Rm| |h|.  \label{eqn:normricupper}
\end{align}
Take trace of (\ref{eqn:normric}), we obtain
\begin{align}
  \D{}{t} H= \frac{1}{2} \Delta H + |h|^2 + \lambda_0 H.  \label{eqn:normscalar}
\end{align}
Define $\displaystyle H_{min}(t) \triangleq \min_{x \in X} H(x,t)$. Apply maximum principle to (\ref{eqn:normscalar}), we obtain
\begin{align}
  \D{}{t} H_{min}(t) \geq \lambda_0 H_{min}(t) \Rightarrow H_{min}(t) \geq e^{\lambda_0 t} H_{min}(0).
  \label{eqn:A16_1}
\end{align}
In particular, the condition $H \geq 0$ is preserved by the normalized Ricci flow (\ref{eqn:normalized}).\\

It follows from (\ref{eqn:normalized}) that the distance derivative with respect to time
is controlled by the $|Ric-\lambda_0 g|$ along the shortest geodesic.
However, a more delicate analysis shows that the lower bound of the distance derivative
depends only on the local Ricci upper bound around the end points.

\begin{proposition}[c.f. section 17 of~\cite{Ha3}, or Lemma 8.3(b) of~\cite{Pe1}]
  Suppose $\left\{ (X, g(t)), 0 \leq t \leq 1 \right\}$
  is a normalized Ricci flow solution $\D{}{t}g=-Ric + \lambda_0 g$ with $|\lambda_0|\leq 1$.
  Suppose $0 \leq t_0 \leq 1$, $x_1, x_2$ are two points in $X$ such that
  $Ric(x, t_0) \leq (m-1)K$ when $d_{g(t_0)}(x, x_1)<r_0$ or $d_{g(t_0)}(x,x_2)<r_0$. Then
  \begin{align}
    \left.  \frac{d}{dt} d_{g(t)}(x_1, x_2) \right|_{t=t_0} \geq
    \frac{1}{2} \lambda_0 d_{g(t_0)}(x_1, x_2) -(m-1)\left( \frac{2}{3} Kr_0 +r_0^{-1} \right).
    \label{eqn:dgoup}
  \end{align}
  \label{prn:dgoup}
\end{proposition}

\begin{proof}
 Without loss of generality, one can assume $t_0=0$.  Then the proof is just an application of the renormalization equation (\ref{eqn:normalization}) and Lemma 8.3(b) of~\cite{Pe1}.
\end{proof}

Suppose $\Omega$ is a compact manifold with boundary. The following lemmas are standard (c.f.~\cite{LiNotes}).

\begin{lemma}
 Suppose $(X, g)$ is a complete manifold, $x_0 \in X$, $0<r\leq 1$. Suppose $r^{-m}|B(x_0,r)| \geq \kappa$ and $r^2 Ric \geq -(m-1)$ in $B(x_0, 2r)$.
 Let $\Omega=B(x_0, r)$. Then the following properties are satisfied.
\begin{itemize}
 \item The isoperimetric constant of $\Omega$ is uniformly bounded by $C_I=C_I(m,\kappa)$.
 \item The Sobolev constant of $\Omega$ is uniformly bounded by $C_S=C_S(m,\kappa)$.
 \item The Neuman Poincar\'e constant of $\Omega$ is uniformly bounded by $C_P=C_P(m,\kappa)$.
\end{itemize}
\label{lma:A18_1}
\end{lemma}

\begin{lemma}
  Suppose $(X, g)$ is a complete Riemannian manifold, $x_0 \in X$. Suppose the following conditions are satisfied.
  \begin{itemize}
    \item For every $0<r<2$, we have $C_{V}^{-1}<|B(x_0, r)|r^{-m}<C_{V}$.
    \item The Sobolev constant of $B(x_0, 2)$ is bounded by $C_S$.
    \item The Poincar\'e constant of $B(x_0, 2)$ is bounded by $C_P$.
    \item $|a|+|\psi|<C_{F}$ on $B(x_0, 2)$.
  \end{itemize}
  Suppose $\varphi \geq 0$ satisfies the inequality $\left( -\Delta + a \right) \varphi \geq \psi$ in the distribution sense, then
  \begin{align}
    \int_{B(x_0, 1)} \varphi \leq C \left( 1+ \inf_{B\left(x_0, \frac{1}{2}\right)} \varphi \right),
    \label{eqn:L31_2}
  \end{align}
  where $C=C(m,C_V,C_S,C_P,C_F)$. Consequently, for every $0<\rho<1$, we have
 \begin{align}
   \rho^{-m}\int_{B(x_0, \rho)} \varphi \leq C \left( \rho^2 + \inf_{B\left(x_0, \frac{\rho}{2}\right)} \varphi \right),
    \label{eqn:L31_3}
  \end{align}
  where $C$ is the same constant as in (\ref{eqn:L31_2}).
  \label{lma:L31_1}
\end{lemma}

\begin{proof}
  Let $\bar{\varphi}=\varphi+C_F$. We compute
  \begin{align*}
    \left( -\Delta + a \right) \bar{\varphi} \geq \psi+aC_F=C_F\left( a+C_F^{-1}\psi \right) \geq
    -C_F\left|a+C_F^{-1}\psi \right|  \geq - \left|a+C_F^{-1}\psi \right| \bar{\varphi} \geq -\left( |a|+1 \right) \bar{\varphi}.
  \end{align*}
  It follows
  $$\Delta \bar{\varphi} \leq \left( 2|a|+1 \right) \bar{\varphi} \leq \left( 2C_F+1 \right) \bar{\varphi}.$$
  By the standard De Giorgi-Nash-Moser iteration (c.f. Lemma 11.2 of~\cite{LiNotes}), we have
  $$\displaystyle \int_{B(x_0, 1)} \bar{\varphi} \leq C \inf_{B\left(x_0, \frac{1}{2} \right)} \bar{\varphi}$$
  for some $C$ depending on $m,C_V,C_S,C_P$ and $C_F$. This in turn implies (\ref{eqn:L31_2}).

  Fix $0<\rho<1$. Let $\tilde{g}=\rho^{-2}g$.  By the scaling property of the Laplacian operator, we see that
  \begin{align*}
   \left( -\Delta + a \right) \varphi \geq \psi \Leftrightarrow
   -\rho^{-2} \Delta_{\tilde{g}} \varphi + a \varphi \geq \psi
   \Leftrightarrow -\Delta_{\tilde{g}} \left( \rho^{-2} \varphi \right) + \rho^2 a \left( \rho^{-2} \varphi \right) \geq \psi.
  \end{align*}
  Let $\tilde{\varphi}=\rho^{-2} \varphi$, we have
  $$\displaystyle -\Delta_{\tilde{g}} \tilde{\varphi} + \rho^2 a \tilde{\varphi} \geq \psi.$$
  Consider this system under the metric $\tilde{g}$. The four estimates hold for this new system, so we obtain
  $$\displaystyle \int_{B_{\tilde{g}}(x_0, 1)} \tilde{\varphi} \leq C \left( 1+ \inf_{B_{\tilde{g}}\left( x_0, \frac{1}{2} \right)} \tilde{\varphi} \right),$$
  which is the same as (\ref{eqn:L31_3}) since $\tilde{\varphi}=\rho^{-2} \varphi$.
\end{proof}

Combing Lemma~\ref{lma:A18_1} and Lemma~\ref{lma:L31_1}, we obtain the following Proposition, which is very useful in the study of boundary estimate.

\begin{proposition}
  Suppose $(X, x_0, g)$ is a complete Riemannian manifold.
  Suppose $0<r \leq 1$,  $r^{-m}|B(x_0,r)| \geq \kappa$,  $r^2 Ric \geq -(m-1)$ on $B(x_0, 2r)$.
  Suppose $\varphi \geq 0$ satisfies the inequality $\left( -\Delta + a \right) \varphi \geq \psi$ for $|a|+|\psi|<C_F$. Then
  for every $0<\rho \leq r$, we have
 \begin{align}
   \rho^{-m}\int_{B(x_0, \rho)} \varphi \leq C \left( \rho^2 + \inf_{B\left(x_0, \frac{\rho}{2}\right)} \varphi \right),
   \label{eqn:A18_1}
  \end{align}
 for some constant $C=C(m,\kappa, C_F)$.
  \label{prn:A18_2}
\end{proposition}

\section{A pseudo-locality theorem}

Under the Ricci flow, an ``almost-Euclidean" region cannot become singular suddenly.  This is the principle of pseudo-locality as stated by Perelman in section 10 of~\cite{Pe1}.
Perelman developed some pseudo-locality theorems by regarding  ``almost'' as close of isoperimetric constant and scalar lower bound.
Of course, this is not the unique ``almost-Euclidean'' condition.   In this section, we will develop similar pseudo-locality properties by explaining  ``almost-Euclidean" balls
as balls whose volume ratio and Ricci lower bound is close to that of the  Euclidean balls'.

\begin{proposition}[A pseudo-locality property, compare Theorem 10.1 and Theorem 10.3 of Perelman~\cite{Pe1}]
For every $0<\alpha<\frac{1}{100m}$, there exist constants $\delta=\delta(m, \alpha), \epsilon=\epsilon(m, \alpha)$ with the following properties.

  Suppose $\left\{ (X, g(t)), 0 \leq t \leq 1 \right\}$ is a Ricci flow solution, $x_0 \in X$. Suppose
  \begin{align}
    &Ric(x, 0) \geq -(m-1)\delta^4,  \quad \forall \; x \in B_{g(0)}\left(x_0,  \delta^{-1} \right). \label{eqn:pseudocondition_1}\\
    &\delta^{m} \left|B_{g(0)}\left(x_0, \delta^{-1} \right) \right|_{d\mu_{g(0)}} \geq (1-\delta) \omega_m.  \label{eqn:pseudocondition_2}
  \end{align}
  Then we have
  \begin{align}
    &\left| B_{g(t)} \left(x, \sqrt{t} \right) \right|_{d\mu_{g(t)}} \geq \kappa' t^{\frac{m}{2}}, \label{eqn:pseudo_kappa} \\
    &|Rm|(x, t) \leq \alpha t^{-1} +\epsilon^{-2}, \quad \forall \; x\in B_{g(0)} \left(x_0, \epsilon \right), \; t \in (0, \epsilon^2],
    \label{eqn:pseudo_ricci}
  \end{align}
  where $\kappa'=\kappa'(m)$ is a universal constant.
  \label{prn:pseudo-locality_L30}
\end{proposition}

\begin{proof}
  We only prove (\ref{eqn:pseudo_ricci}).  The proof of (\ref{eqn:pseudo_kappa}) follows verbatim.

If the statement was false, we can find a sequence of $\delta_k, \epsilon_k \to 0$, $x_k \in X_k$ such that (\ref{eqn:pseudocondition_1}) and (\ref{eqn:pseudocondition_2}) hold.
However, (\ref{eqn:pseudo_ricci}) are violated.

Following the proof of Perelman's pseudo-locality theorem, we can find a sequence of functions $u_k$ which are compactly supported on $B(x_k, 1)$ and satisfy
(See the end of the proof of Theorem 10.1 of~\cite{Pe1}):
 \begin{align}
   & \int_{B(x_k, 1)} u_k  =1, \label{eqn:L19_1} \\
   & \int_{B(x_k, 1)} \left\{ \frac{1}{2} \left|\nabla f_k \right|^2 +f_k -m \right\} u_k \leq -\eta<0,  \label{eqn:L19_2}
 \end{align}
 where $u_k = (2\pi)^{-\frac{m}{2}} e^{-f_k}$. Of course, here we regard $d\mu_{g_k(0)}$ as the default measure.   Let $\bar{u}_k=\sqrt{u_k}$.  These equations can be written as
 \begin{align*}
   & \int_{B(x_k, 1)} \bar{u}_k^2 =1, \\
   & \int_{B(x_k, 1)} \left\{ 2|\nabla \bar{u}_k|^2 - 2 \bar{u}_k^2\log \bar{u}_k  - m\left( 1+\log \sqrt{2\pi} \right) \bar{u}_k^2 \right\}  \leq -\eta.
 \end{align*}
 Denote by $\mathcal{F}_k(u_k)$ the integral
 $$\int_{B(x_k, 1)} \left\{ 2|\nabla \bar{u}_k|^2 - 2 \bar{u}_k^2\log \bar{u}_k  -m\left( 1+\log \sqrt{2\pi} \right) \bar{u}_k^2 \right\}.$$
 Clearly, $\mathcal{F}_k$ is a functional on the
 space of functions $\bar{u} \in W_{0}^{1,2}(B(x_k, 1))$ satisfying $\int_{B(x_k, 1)} \bar{u}^2 =1$. By the result of Rothaus (\cite{ROS}),
 we see that $\mathcal{F}_k$ has a minimizer $\varphi_k$, which satisfies the Euler-Lagrange equation
 \begin{align}
     -2\Delta \varphi_k - 2 \varphi_k \log \varphi_k -m\left(1+\log \sqrt{2\pi} \right)\varphi_k = \lambda_k \varphi_k.
     \label{eqn:L19_3}
 \end{align}
 On one hand, by the choice of $\lambda_k$, we have
 $$\displaystyle \lambda_k = \mathcal{F}(\varphi_k) \leq \mathcal{F}(\bar{u}_k) \leq -\eta<0.$$
 On the other hand, integrating (\ref{eqn:L19_3}) over $B(x_k, 1)$ implies
 \begin{align}
   &\quad \lambda_k+m\left( 1+\log \sqrt{2\pi} \right) \nonumber\\
   &=\int 2|\nabla \varphi_k|^2 - 2 \int \varphi_k^2 \log \varphi_k \nonumber\\
   &\geq \int 2|\nabla \varphi_k|^2 -2\int \varphi_k^{2} \cdot \left( \frac{m}{2e}\varphi_k^{\frac{2}{m}} \right)
    \nonumber\\
   &=\int 2|\nabla \varphi_k|^2 - \frac{m}{e} \int \varphi_k \cdot \varphi_k^{\frac{m+2}{m}}. \label{eqn:B8_1}
 \end{align}
 In the third step, we used the fact $\log x \leq \frac{m}{2e}x^{\frac{2}{m}}$ for every positive $x$.
 Plug H\"older inequality into (\ref{eqn:B8_1}) yields
 \begin{align}
   &\quad \lambda_k+m\left( 1+\log \sqrt{2\pi} \right) \nonumber\\
   &\geq \int 2|\nabla \varphi_k|^2
   -\frac{m}{e}  \left(\int \varphi_k^{\frac{2m}{m-2}} \right)^{\frac{m-2}{2m}} \cdot \left( \int \varphi_k^2 \right)^{\frac{m+2}{2m}}
    \nonumber\\
   &=\int 2|\nabla \varphi_k|^2
   -2 \cdot \frac{m}{2e}  \left(\int \varphi_k^{\frac{2m}{m-2}} \right)^{\frac{m-2}{2m}} \nonumber\\
   &\geq \int 2|\nabla \varphi_k|^2 - \left\{ a^2 \left(\int \varphi_k^{\frac{2m}{m-2}} \right)^{\frac{m-2}{m}} +\frac{m^2}{4a^2e^2}  \right\},
   \label{eqn:L30_1}
 \end{align}
 where $a$ is a positive constant to be determined.  Apply Lemma~\ref{lma:A18_1}, we obtain uniform bound for the Sobolev constant of $B(x_k, 1)$.
 It follows that
 \begin{align}
   \left(\int_{B(x_k, 1)} \varphi_k^{\frac{2m}{m-2}} \right)^{\frac{m-2}{m}} \leq C_S \int_{B(x_k, 1)} \left( \varphi_k^2 + |\nabla \varphi_k|^2 \right).
   \label{eqn:L30_2}
 \end{align}
 Let $a^2=\frac{2}{C_S}$ and put (\ref{eqn:L30_2}) into (\ref{eqn:L30_1}), we obtain
 \begin{align}
   \lambda_k+m\left( 1+\log \sqrt{2\pi} \right)
   \geq  (2-a^2C_S) \int |\nabla \varphi_k|^2- \left( a^2C_S +\frac{m^2}{4a^2e^2} \right)=-\left( 2+\frac{m^2C_S}{8e^2} \right).
   \label{eqn:L30_3}
 \end{align}
 Recall that $\lambda_k \leq -\eta <0$,  from (\ref{eqn:L30_3}) we see that there exists a constant $C_{\lambda}$, which depends on $m, C_S$, such that
 \begin{align}
   |\lambda_k|< C_{\lambda}.      \label{eqn:L30_4}
 \end{align}
 Note that the Euler-Lagrangian equation of $\varphi_k$ can be written as
 \begin{align}
   -\Delta \varphi_k = \left( \frac{1}{2}\left(m+m\log \sqrt{2\pi}+\lambda_k\right) + \log \varphi_k \right) \varphi_k.
 \label{eqn:L29_4}
 \end{align}
 Define $\bar{\varphi}_k \triangleq \max \left\{ \varphi_k, 1 \right\}$.  Since $\log x \leq \frac{m}{2e} x^{\frac{2}{m}}$ for every $x>0$, it follows
 from (\ref{eqn:L29_4}) that $\bar{\varphi}_k$ satisfies the inequality
 \begin{align}
   -\Delta \bar{\varphi}_k \leq \frac{1}{2} \left( m+ m\log \sqrt{2\pi}+\lambda_k + \frac{m}{e} \bar{\varphi}_k^{\frac{2}{m}} \right) \bar{\varphi}_k
   \label{eqn:L30_5}
 \end{align}
 in the distribution sense.  Clearly, we can uniformly bound the $L^{m}(B(x_k,2))$-norm of
$$\displaystyle \frac{1}{2} \left(m+ m\log \sqrt{2\pi}+\lambda_k + \frac{m}{e} \bar{\varphi}_k^{\frac{2}{m}} \right),$$
where $m>\frac{m}{2}$. Note that $B(x_k, 2)$ has a uniform Sobolev constant $C_S$. Then the standard Moser iteration implies that
$$\displaystyle \norm{\bar{\varphi}_k}{C^0(B(x_k,1))} \leq C \norm{\bar{\varphi}_k}{L^2(B(x_k,2))} \leq C,$$
which in turn implies
 \begin{align}
   \norm{\varphi_k}{C^0(B(x_k,1))} \leq C_1=C_1(m, C_S, C_{\lambda}).
 \label{eqn:L19_4}
 \end{align}
 Recall that Ricci curvature is uniformly bounded from below on $B(x_k, 2)$,
 the estimate of Cheng-Yau (c.f.~\cite{ChengYau}, section 6 of~\cite{LiNotes}) implies that
 \begin{align}
   |\nabla \varphi_k(x)| \leq C_2(m, d(x, \partial B(x_k, 1))), \quad \forall x \in B(x_k, 1).
   \label{eqn:L19_5}
 \end{align}
 In view of the non-collapsed condition and Ricci lower bound, we have the convergence in the pointed Gromov-Hausdorff topology,
 \begin{align}
   \left( X_k, x_k, g_k(0) \right) \longright{Gromov-Hausdorff}  \left( X_{\infty}, x_{\infty}, g_{\infty} \right).
   \label{eqn:L29_7}
 \end{align}
 Combining (\ref{eqn:L19_4}), (\ref{eqn:L19_5}) and (\ref{eqn:L29_7}), we obtain a locally-Lipschitz limit function $\varphi_{\infty}$ on
 $B(x_{\infty}, 1) \subset X_{\infty}$ with $\displaystyle \norm{\varphi_{\infty}}{C^0(B(x_{\infty}, 1))} \leq C_1$.
 In general, it is hard to expect $\varphi_{\infty}$ to be better than a locally-Lipschitz function.
 However, by Theorem 0.8 of~\cite{Colding_volume}, we know that $X_{\infty}$ is isometric to the Euclidean space $\left( \R^m, g_{\E} \right)$,
 which has a lot of excellent properties. We will use these properties to show that $\varphi_{\infty}$ has much better regularity than a general locally Lipshitz function.\\

 \begin{claim}
   $\varphi_{\infty}$ can be extended to be a continuous function defined on $\overline{B(x_{\infty}, 1)}$ with
   \begin{align}
     \varphi_{\infty}|_{\partial B(x_{\infty}, 1)}=0.
   \label{eqn:clmL19_1}
   \end{align}
   \label{clm:L19_1}
 \end{claim}

 It suffices to show  $\displaystyle \lim_{r \to 0} \norm{\varphi_{\infty}}{L^{\infty}(B(w,r))}=0$ for arbitrary $w \in \partial B(x_{\infty}, 1)$.

 Fix arbitrary $w \in \partial B(x_{\infty}, 1)$. Suppose $w_k \in \partial B(x_k, 1)$ and $w_k \to w$ as $X_k$ converges to $X_{\infty}$.
 For brevity, define  $M_{d,k} \triangleq Osc_{B\left( w_k, d \right)}(\varphi_k)$. By trivial extension, we can look $\varphi_k$ as a function defined on the whole manifold $X_k$.
 Then define $\psi_{d,k} \triangleq M_{2d,k}-\varphi_k$.  In view of (\ref{eqn:L29_4}), it is easy to see that $\psi_{d,k}$ satisfies the inequality
 \begin{align}
   &\quad \left( -\Delta -\frac{1}{2} \left( m+ m\log \sqrt{2\pi} +\lambda_k \right) \right) \psi_{d,k}   \nonumber\\
   &= -\frac{M_{2d,k}\left(m+ m\log \sqrt{2\pi}+\lambda_k \right)}{2}- (M_{2d,k}-\psi_{d,k}) \log \left( M_{2d,k}-\psi_{d,k} \right) \nonumber\\
   &\geq -C_3=-C_3(m, C_S, C_{\lambda})
   \label{eqn:L29_6}
 \end{align}
 in the sense of distribution. In other words, $\psi_{d,k}$ is a super-solution of the corresponding elliptic system.   Clearly, in the ball
 $B(w_k, 4d) \subset B(x_k, 10)$,  every geodesic ball's volume ratio is bounded from two sides.
 Apply Proposition~\ref{prn:A18_1}, we obtain
 \begin{align}
   (2d)^{-m}\int_{B(w_k, 2d)} \psi_{d,k} \leq C_4 \left( \inf_{B(w_k, d)} \psi_{d,k} + d^2 \right).
   \label{eqn:L31_6}
 \end{align}

 \begin{figure}[h]
 \begin{center}
    \psfrag{xk}[c][c]{$x_k$}
    \psfrag{xif}[c][c]{$x_{\infty}$}
    \psfrag{wk}[c][c]{$w_k$}
    \psfrag{wif}[c][c]{$w$}
    \psfrag{GH}[c][c]{G.H.convergence}
    \psfrag{Bk}[l][l]{$B(x_k,1)$}
    \psfrag{Bif}[l][l]{$B(x_{\infty},1)$}
    \psfrag{SBk}[c][c]{$B(w_k,2d)$}
    \psfrag{SBif}[l][l]{$B(w,2d)$}
    \includegraphics[width=0.8\columnwidth]{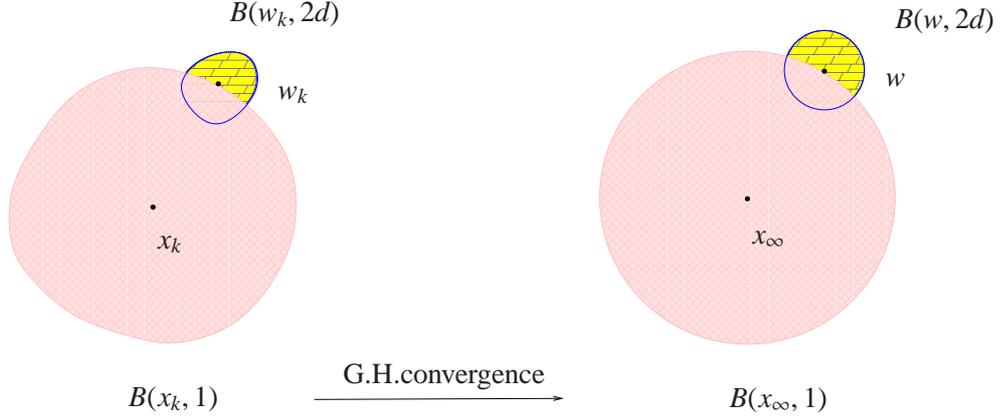}
  \caption{Boundary estimates}
  \label{fig:Boundary}
 \end{center}
\end{figure}

 By the volume continuity, it is not hard (Figure~\ref{fig:Boundary}) to see that the volume of $B(w_k, 2d) \backslash  B(x_k, 1)$
 is strictly greater than a fixed portion of the volume of
 $B(w_k, 2d)$, which is almost $\omega_m (2d)^m$.   For brevity, let's say $\left|B(w_k, 2d) \backslash  B(x_k, 1) \right|> 10^{-m} \cdot \omega_m (2d)^m$.
 Put this into (\ref{eqn:L31_6}) and note that $\displaystyle \inf_{B(w_k,d)} \psi_{d,k}=M_{2d,k}-M_{d,k}$, we have
 $\displaystyle 10^{-m} \omega_m M_{2d,k} < C_4 \left(M_{2d,k}-M_{d,k} +d^2 \right)$, which implies
 \begin{align}
   M_{d,k} <  \left( 1- 10^{-m}C_4^{-1} \omega_4 \right) M_{2d,k} + d^2 \triangleq \gamma M_{2d,k} +d^2.
   \label{eqn:L31_8}
 \end{align}
 By choosing $C_4$ large, we can assume $\gamma \in (0, 1)$.
 Let $d=2^{-i}$, $i>1$.  Induction of (\ref{eqn:L31_8}) yields
\begin{align*}
  M_{2^{-i},k} < \gamma M_{2^{-i+1},k} + 4^{-i} < \gamma^{i-1}M_{\frac{1}{2},k} + \sum_{j=0}^{i-2} \gamma^j 4^{-i+j}
  =\gamma^{i-1}M_{\frac{1}{2},k}+ \frac{\gamma^{i-1}-4^{-i+1}}{4(4\gamma-1)}.
\end{align*}
Recall that $M_{\frac{1}{2},k} \leq \norm{\varphi_k}{B(x_k,1)} \leq C_1$. Let $k \to \infty$, we obtain
 \begin{align}
   \norm{\varphi_{\infty}}{L^{\infty}(B(w,2^{-i}))} \leq  \lim_{k \to \infty} M_{2^{-i+1},k} \leq C_1\gamma^{i-1}+\frac{\gamma^{i-1}-4^{-i+1}}{4(4\gamma-1)}.
   \label{eqn:L29_3}
 \end{align}
 Since $\gamma \in (0, 1)$,  it is clear that (\ref{eqn:L29_3}) implies $\displaystyle \lim_{r \to 0} \norm{\varphi_{\infty}}{L^{\infty}(B(w,r))}=0$.
 So we finish the proof of Claim~\ref{clm:L19_1}. \\

 \begin{claim}
   In $B(x_{\infty}, 1)$, $\varphi_{\infty}$ satisfies the following equation
   \begin{align}
     -2\Delta \varphi_{\infty} -2 \varphi_{\infty} \log \varphi_{\infty} - \left(m+m\log \sqrt{2\pi}+\lambda_{\infty} \right) \varphi_{\infty}=0.
     \label{eqn:clmL19_2}
   \end{align}
   Consequently, $\varphi_{\infty} \in C^{\infty}(B(x_{\infty}, 1))$.
   \label{clm:L19_2}
 \end{claim}

  Note that $B(x_{\infty}, 1)$ is a unit ball in the standard $\R^m$.  In particular, it has smooth boundary.
 So equation (\ref{eqn:clmL19_2}) is equivalent to the following integration equation
 \begin{align}
   \varphi_{\infty}(z) = \int_{B(x_{\infty}, 1)} G(z, y)  \left( \frac{m+m\log \sqrt{2\pi}+\lambda_{\infty}}{2} + \log \varphi_{\infty}(y) \right) \varphi_{\infty}(y)dy,
   \label{eqn:L19_int}
 \end{align}
for every $z \in B(x_{\infty}, 1)$. Here $G$ is the Green function of the unit ball $B(x_{\infty}, 1) \subset \R^m$.
Because $B(x_{\infty}, 1)$ is simple, we can write down $G(z, y)$ explicitly,
 \begin{align*}
   G(z,y)=\frac{1}{(m-2)m\omega_m} \left( d^{2-m}(z,y) - d^{2-m}(x_{\infty}, z)d^{2-m}(z^*, y)\right),
 \end{align*}
 whenever $z \neq y$.  Here $z^*$ is the symmetric point of $z$ with respect to $\partial B(x_{\infty}, 1)$.
 If $z \neq x_{\infty}$, $z^*$ is the point such that $x_{\infty}, z, z^*$ on the same straight line
 and $\left|\overline{x_{\infty} z} \right| \cdot \left|\overline{x_{\infty} z^*} \right|=1$.
 If $z=x_{\infty}$, we assume $z^{*}$ as the infinity point. In the later case, we have
 $$G(x_{\infty}, y)=\frac{1}{(m-2)m\omega_m} \left(d^{2-m}(x_{\infty}, y)-1\right).$$

 By continuity, for proving (\ref{eqn:L19_int}) in $B(z_{\infty}, 1)$, it suffices to show (\ref{eqn:L19_int}) for every
 $z \in B(x_{\infty}, 1) \backslash \left\{ x_{\infty} \right\}$.  Without loss of generality, we fix an arbitrary point
 $z \in B(x_{\infty}, 1) \backslash \left\{ x_{\infty} \right\}$.
 Suppose $z_k \in B(x_k, 1)$ and $z_k \to z$, $z_k^{*} \in X_k$ and $z_k^{*} \to z^*$ (See Figure~\ref{fig:Greenfunction}).

 \begin{figure}[h]
 \begin{center}
    \psfrag{xk}[c][c]{$x_k$}
    \psfrag{yk}[c][c]{$y_k$}
    \psfrag{zk}[c][c]{$z_k$}
    \psfrag{xif}[c][c]{$x_{\infty}$}
    \psfrag{yif}[c][c]{$y$}
    \psfrag{zif}[c][c]{$z$}
    \psfrag{zk*}[c][c]{$z_k^*$}
    \psfrag{zif*}[c][c]{$z^*$}
    \psfrag{GH}[c][c]{G.H.convergence}
    \psfrag{Bk}[l][l]{$B(x_k,1)$}
    \psfrag{Bif}[l][l]{$B(x_{\infty},1)$}
    \includegraphics[width=0.8\columnwidth]{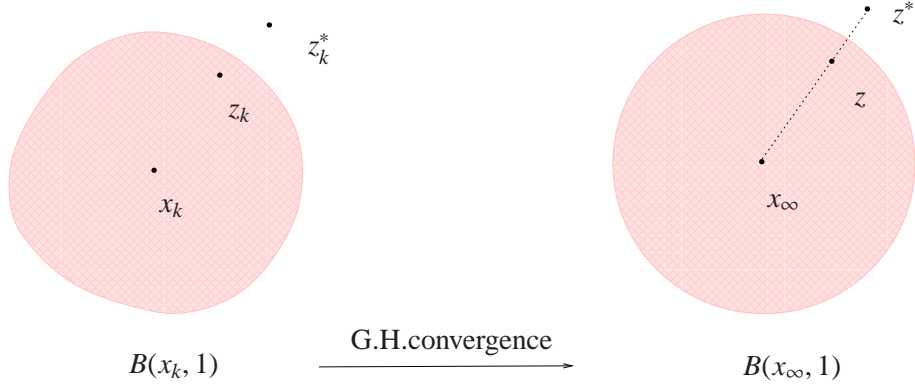}
  \caption{Approximation of Green functions}
  \label{fig:Greenfunction}
 \end{center}
\end{figure}

 Let $d$ be the distance function to the point $z_k$ under the metric $g_k(0)$. Note that
 \begin{align}
   \Delta d^{2-m}=(2-m)d^{-m} \left( 1-m+ d\Delta d \right).  \label{eqn:L20_1}
 \end{align}
 If the underlying space is Euclidean, then the right hand side is equal to $0$ whenever $d>0$.  Now on $X_k$, we are focusing our attention around the point $x_k$,
 where $Ric \geq -(m-1) \delta_k^4$. Clearly, Laplacian comparison theorem (c.f. Corollary 1.131 of~\cite{CLN}) implies that
 \begin{align}
   \Delta d^{2-m} +(m-2)(m-1) d^{1-m} \delta_k^2 \geq 0 \label{eqn:L20_2}
 \end{align}
 on $B(x_k, 10)$. It follows that
 \begin{align*}
   &0 \leq \int_{B(x_k, 1) \backslash B(z_k, r)} \left\{ \Delta d^{2-m} + (m-2)(m-1)d^{1-m}\delta_k^2 \right\}\\
   &\leq \int_{B(z_k, 2) \backslash B(z_k, r)} \left\{ \Delta d^{2-m} + (m-2)(m-1)d^{1-m}\delta_k^2 \right\}\\
   &=(m-2) \left\{ |\partial B(z_k, 2)|2^{1-m}-|\partial B(z_k, r)|r^{1-m} \right\}
   + (m-2) (m-1)\delta_k^2 \int_{r}^{2}\left( \rho^{1-m} |\partial B(z_k, \rho)| \right)  d\rho \\
       & < (m-2) \left\{ \left|  |\partial B(z_k, 2)|2^{1-m} -|\partial B(z_k, r)| r^{1-m}\right| + 4(m-1)m\omega_m\delta_k^2\right\}.
 \end{align*}
 Consequently, we have
 \begin{align*}
    &\quad \int_{B(x_k, 1) \backslash B(z_k, r)} \left|\Delta d^{2-m} \right| \\
    &\leq   \int_{B(x_k,1) \backslash B(z_k, r)} \left|\Delta d^{2-m} + (m-2) (m-1)d^{1-m}\delta_k^2 \right|
     + \int_{B(x_k, 1) \backslash B(z_k, r)} (m-2) (m-1) d^{1-m}\delta_k^2\\
     &< (m-2) \left\{ \left|  |\partial B(z_k, 2)|2^{1-m} - |\partial B(z_k, r)| r^{1-m}\right| + 8(m-1)m\omega_m \delta_k^2 \right\}.
 \end{align*}
 Fix $k$, let $r \to 0$, we have
 \begin{align*}
   \int_{B(x_k, 1) \backslash \left\{ z_k \right\}} \left| \Delta d^{2-m}\right|
   \leq (m-2) \left\{ \left|  |\partial B(z_k, 2)|2^{1-m} - m\omega_m \right| + 8(m-1)m\omega_m \delta_k^2\right\}.
 \end{align*}
 Therefore, we obtain
 \begin{align}
   \int_{B(x_k, 1) \backslash \left\{ z_k \right\}} \left| \varphi_k \Delta d^{2-m} \right|
    \leq C_1 (m-2) \left\{ \left|  |\partial B(z_k, 2)|2^{1-m} - m \omega_m \right| + 8(m-1)m\omega_m \delta_k^2\right\} \to 0,
    \label{eqn:L19_11}
 \end{align}
 as $k \to \infty$, since the limit space $X_{\infty}$ is Euclidean, where every geodesic sphere has the same volume ratio: $m\omega_m$.
 Consequently, we can calculate
 \begin{align}
   \int_{B(x_k,1)}  d^{2-m}(z_k,y) \Delta \varphi_k(y) dy&= \int_{B(x_k,1) \backslash \left\{ z_k \right\}}  d^{2-m}(z_k,y) \Delta \varphi_k(y) dy\nonumber\\
   &=\lim_{r \to 0} \int_{B(x_k, 1) \backslash B(z_k, r)} d^{2-m}(z_k,y) \Delta \varphi_k(y) dy\nonumber\\
    &=(m-2)m\omega_m \varphi_k(z_k) + \int_{B(x_k, 1) \backslash \left\{ z_k \right\}} \varphi_k(y) \Delta d^{2-m}(z_k, y) dy.
   \label{eqn:L19_7}
 \end{align}
 Of course, the default measure in the calculation is $dy=d\mu_{g_k(0)}$. Combining (\ref{eqn:L19_11}) and (\ref{eqn:L19_7}), we have
 \begin{align}
   \lim_{k \to \infty}  \int_{B(x_k,1)}  d^{2-m}(z_k,y)\Delta \varphi_k(y) dy=  (m-2)m\omega_m \varphi_{\infty}(z).
   \label{eqn:L19_8}
 \end{align}
 Note that $d(z_k^*, \cdot)>0$ uniformly on $B(x_k, 1)$. By similar but simpler arguments, we obtain
 \begin{align}
   \lim_{k \to \infty}  \int_{B(x_k,1)}  d^{2-m}(z_k^*,y) \Delta \varphi_k(y) dy= 0.
   \label{eqn:L19_9}
 \end{align}
 In view of (\ref{eqn:L19_8}) and (\ref{eqn:L19_9}), we have
 \begin{align*}
   &\quad (m-2)m\omega_m \varphi_{\infty}(z)\\
   &=\lim_{k \to \infty} \int_{B(x_k, 1)} \left\{ d^{2-m}(z_k, y) -d^{2-m}(x_k, z_k)d^{2-m}(z_k^*, y)\right\} \Delta \varphi_k(y) dy\\
   &=\lim_{k \to \infty} \int_{B(x_k, 1)}  \left\{ d^{2-m}(z_k, y) -d^{2-m}(x_k, z_k)d^{2-m}(z_k^*, y) \right\}
    \left(  \frac{m+m\log \sqrt{2\pi}+\lambda_k}{2} + \log \varphi_k(y)\right) \varphi_k(y) dy\\
   &=\int_{B(x_{\infty}, 1)} \left\{ d^{2-m}(z,y) -d^{2-m}(x_{\infty}, z)d^{2-m}(z^*,y) \right\} \cdot
   \left( \frac{m+m\log \sqrt{2\pi}+\lambda_{\infty}}{2} + \log \varphi_{\infty}(y) \right) \varphi_{\infty}(y) dy\\
   &= (m-2)m\omega_m \int_{B(x_{\infty}, 1)} G(z, y)  \cdot \left( \frac{m+m\log \sqrt{2\pi}+\lambda_{\infty}}{2} + \log \varphi_{\infty}(y) \right) \varphi_{\infty}(y) dy.
 \end{align*}
 In the third step, we used the Euler-Lagrangian equation for $\varphi_k$.  In the fourth step, we used the integrability of $d^{2-m}$ and uniform bound of $\varphi_k$.
 Therefore, we prove (\ref{eqn:L19_int}) for $z$.  By the arbitrariness of $z \in B(x_{\infty}, 1) \backslash \left\{ x_{\infty} \right\}$ and continuity, equation (\ref{eqn:L19_int}),
 henceforth (\ref{eqn:clmL19_2}) follows directly. Then the standard bootstrapping argument for elliptic PDEs implies that $\varphi_{\infty} \in C^{\infty}\left( B(x_{\infty}, 1) \right)$. This finishes the proof of Claim~\ref{clm:L19_2}.\\

 Now we are ready to prove the theorem by a contradiction argument. In fact, since $\partial B(x_{\infty}, 1)$ is smooth
 and $\varphi|_{\partial B(x_{\infty}, 1)} \equiv 0$ (Claim~\ref{clm:L19_1}), by trivial extension, we can regard $\varphi_{\infty} \in W_0^{1,2}(\R^m)$ (c.f. Section 5.5 of~\cite{Evans}). It follows from the Logarithm Sobolev inequality of Euclidean space (c.f.~\cite{Gross}) that
 \begin{align}
   \int_{\R^m} \left( \frac{1}{2} |\nabla \varphi_{\infty}|^2 - 2\varphi_{\infty}^2 \log \varphi_{\infty}
   -m \left(1+\log \sqrt{2\pi} \right)\varphi_{\infty}^2 \right)  \geq 0.
   \label{eqn:L19_6}
 \end{align}
 On the other hand, by (\ref{eqn:clmL19_2}) in Claim~\ref{clm:L19_2} and the fact $\varphi_{\infty} \equiv 0$ outside $B(x_{\infty}, 1)$, we deduce that
 \begin{align*}
   \int_{\R^m} \left( \frac{1}{2} |\nabla \varphi_{\infty}|^2 - 2\varphi_{\infty}^2 \log \varphi_{\infty}
   -m \left(1+\log \sqrt{2\pi} \right)\varphi_{\infty}^2 \right)=\lambda_{\infty} \leq -\eta<0,
 \end{align*}
 which contradicts to (\ref{eqn:L19_6})!
\end{proof}

\begin{remark}
  If the ``almost-Euclidean volume ratio" (inequality (\ref{eqn:pseudocondition_2}))
  and ``almost nonnegative Ricci" (inequality (\ref{eqn:pseudocondition_1})) hold globally,
  then the rough curvature estimate (inequality (\ref{eqn:pseudo_ricci}))
  follows from the combination of Perelman's pseudo-locality theorem and Levy-Gromov inequality(c.f.~\cite{Gromov}), whose proof requires some regularity results
  in geometric measure theory on closed manifolds.
  There should exist another proof of Proposition~\ref{prn:pseudo-locality_L30} from some local version of the Gromov-Ivey inequality.
  However, it seems that some local regularity results in geometric measure theory are required.
\end{remark}

\begin{remark}
  Except inequality (\ref{eqn:pseudocondition_1}),
  the ``almost nonnegative Ricci" condition can also be interpreted as the $L^p$-integration of negative Ricci part
  is sufficiently small(c.f.~\cite{PW1},~\cite{PW2}), for some $p>\frac{m}{2}$. Using this interpretation, one can obtain another pseudo-locality theorem.
\end{remark}

Combine Proposition~\ref{prn:pseudo-locality_L30} with the fundamental work of~\cite{CC1}, we obtain the following property.

\begin{proposition}
There exists a constant $\delta_0=\delta_0(m)$ with the following properties.

Suppose $\left\{ (X, g(t)), 0 \leq t \leq 1 \right\}$ is a Ricci flow solution, $x_0 \in X, \Omega=B_{g(0)}(x_0, 1)$. Suppose that
  \begin{align}
    Ric(x, 0) \geq -(m-1)\delta_0, \quad \forall \; x \in \Omega; \quad |\Omega|_{d\mu_{g(0)}} \geq (1-\delta_0) \omega_m.
   \label{eqn:A21_5}
  \end{align}
  Then we have
  \begin{align*}
    \left|B_{g(s)} \left(x, \sqrt{s} \right) \right|_{d\mu_{g(s)}} \geq \kappa' s^{\frac{m}{2}}, \quad
    \left|\widetilde{Rm} \right|(x, s) \leq \frac{1}{100}s^{-1}, \quad \forall \; x\in \Omega'=B_{g(0)} \left(x_0, \frac{3}{4} \right), \; s \in (0, 2\delta_0],
  \end{align*}
  where $\kappa'=\kappa'(m)$ is a universal constant.
  \label{prn:pseudo-locality_A21}
\end{proposition}

\begin{proof}
  Let's first prove the following Claim.
   \begin{claim}
   For every small $\xi>0$, there exists a number $\eta=\eta(m, \xi)$ with the following property.

    Suppose $Ric(x,0) \geq -(m-1)\eta$ in $\Omega=B_{g(0)}(x_0, 1)$, and $|\Omega|_{d\mu_{g(0)}} \geq (1-\eta) \omega_m$, then
    \begin{align}
      8^m \left|B_{g(0)}\left( y, \frac{1}{8} \right) \right|_{d\mu_{g(0)}} \geq (1-\xi) \omega_m, \quad \forall \; y \in B_{g(0)}\left(x_0, \frac{3}{4} \right).
    \label{eqn:A21_7}
    \end{align}
    \label{clm:A21_1}
  \end{claim}

  Actually, if this statement was wrong, we can find a sequence of $\eta_i \to 0$ and manifolds $(X_i, x_i, g_i(0))$ such that (\ref{eqn:A21_5}) holds for
  $\eta_i$ and the ball $\Omega_i=B_{g_i(0)}(x_i, 1)$. However, for some point $y_i \in B_{g_i(0)}\left( x_i, \frac{3}{4} \right)$, we have
  \begin{align}
   8^m \left|B_{g_i(0)}\left( y_i, \frac{1}{8} \right) \right|_{d\mu_{g_i(0)}} < \left( 1-\xi \right) \omega_m.
   \label{eqn:A21_6}
  \end{align}
  Suppose $\left( \Omega_i, x_i, g_i(0) \right)$ converges to $\left( \bar{\Omega}, \bar{x}, \bar{g} \right)$.  Clearly, we see that $\bar{\Omega}$ is isometric to
  the unit ball in the Euclidean space  $\R^m$. Since $y_i \in B_{g_i(0)}\left(x_i, \frac{3}{4} \right)$,
  we can assume $y_i \to \bar{y} \in B\left(\bar{x}, \frac{3}{4} +\frac{1}{100} \right) \subset \bar{\Omega}$. The lower bound of Ricci guarantees the continuity of volume. Therefore we have
  \begin{align*}
    \lim_{i \to \infty} 8^m \left|B_{g_i(0)}\left( y_i, \frac{1}{8} \right) \right|_{d\mu_{g_i(0)}}
    =8^m \left|B_{\bar{g}}\left( \bar{y}, \frac{1}{8} \right) \right|_{d\mu_{\bar{g}}}=\omega_m,
  \end{align*}
  which contradicts to (\ref{eqn:A21_6})! This contradiction establishes the proof of Claim~\ref{clm:A21_1}.   \\

  Let $\xi=\delta^4 \left(m, \frac{1}{1000m} \right)$, where $\delta$ is defined by Proposition~\ref{prn:pseudo-locality_L30}.
  Let $\eta=\eta(m,\xi)$ according to Claim~\ref{clm:A21_1}.

  Suppose the conditions of Claim~\ref{clm:A21_1} is satisfied for $\eta=\eta(m,\xi)$.
  Define $\hat{g}(t)=\xi^{-2} g(\xi^2 t)$. Fix an arbitrary point $y \in B_{g(0)}\left(x_0, \frac{3}{4} \right)$.
  By volume comparison, inequality (\ref{eqn:A21_7}) and the choice of $\xi$ yield that
  $(X, y, \hat{g}(0))$ satisfies the initial conditions of Proposition~\ref{prn:pseudo-locality_L30}.
  In particular, we have
$\displaystyle \left|B_{\hat{g}(t)}\left(y, \sqrt{t}\right) \right|_{d\mu_{\hat{g}(t)}} \geq \kappa' t^{\frac{m}{2}}, \quad
    |Rm|_{\hat{g}(t)}(y) \leq \frac{1}{1000m t} + \epsilon^{-2}$ for every $t \in \left(0, \epsilon^2 \right]$.
This implies that for every $t \in \left(0, \frac{\epsilon^2}{200} \right]$, we have
$$\displaystyle \left|B_{\hat{g}(t)}\left(y, \sqrt{t}\right) \right|_{d\mu_{\hat{g}(t)}} \geq \kappa' t^{\frac{m}{2}}, \quad
     |Rm|_{\hat{g}(t)}(y) \leq \frac{1}{100t}.$$
By a trivial rescaling argument, we conclude
  \begin{align*}
    \left|B_{g(t)}\left(y, \sqrt{t}\right) \right|_{d\mu_{g(t)}} \geq \kappa' t^{\frac{m}{2}}, \quad
    |Rm|_{g(t)}(y) \leq \frac{1}{100 t}, \quad \forall \; t \in \left( 0, \frac{\xi^2 \epsilon^2}{200} \right].
  \end{align*}
  Define $\delta_0 \triangleq \min \left\{ \frac{\xi^2 \epsilon^2}{1000}, \eta(m,\xi) \right\}$. Clearly, Proposition~\ref{prn:pseudo-locality_A21} holds
  for this choice of $\delta_0$.
\end{proof}

Now we are ready to prove  the pseudo-locality theorem under the normalized Ricci flow.

\begin{theorem}[Pseudo-locality theorem]
  There exists a constant $\delta_0=\delta_0(m)$ with the following properties.

  Suppose $\left\{ (X, g(t)), 0 \leq t \leq 1 \right\}$ is a normalized Ricci flow solution: $\D{}{t} g= -Ric + \lambda_0 g$, $\lambda_0$ is a constant
  with $|\lambda_0| \leq 1$. Let $x_0 \in X, \Omega=B_{g(0)}(x_0, 1)$. Suppose that
  \begin{align}
    Ric(x, 0) \geq -(m-1)\delta_0, \quad \forall \; x \in \Omega; \quad |\Omega|_{d\mu_{g(0)}} \geq (1-\delta_0) \omega_m.
   \label{eqn:pseudocon}
  \end{align}
  Then we have
  \begin{align}
    &\left| B_{g(t)} \left(x, \sqrt{t} \right) \right|_{d\mu_{g(t)}} \geq \kappa_0 t^{\frac{m}{2}}, \label{eqn:pseduo_noncollapsed}\\
    &|Rm|(x, t) \leq t^{-1}, \quad \forall \; x\in \Omega'=B_{g(0)} \left(x_0, \frac{3}{4} \right), \; t \in (0, 2\delta_0],
    \label{eqn:pseudo}
  \end{align}
  where $\kappa_0=\kappa_0(m)$ is a universal constant.
  \label{thm:pseudo-locality_A22}
\end{theorem}

\begin{proof}
    Let $\tilde{g}(s)= \left( 1-2\lambda_0 s \right) g\left( \frac{\log (1-2\lambda_0 s)}{-\lambda_0} \right)$.
    Clearly, $\tilde{g}(s)$ is a Ricci flow solution with $\tilde{g}(0)=g(0)$.
    Denote $\frac{\log (1-2\lambda_0 s)}{-\lambda_0}$ by $t(s)$. Then we have $\tilde{g}(s)=\left( 1-2\lambda_0 s \right) g(t)$.
  By Taylor expansion of $t(s)=\frac{\log (1-2\lambda_0 s)}{-\lambda_0}$, shrink $\delta_0$ if necessary, we have
  $ \frac{3}{2} s < t < 3s$ whenever $s \in (0, 10\delta_0)$.
  Note that
  $$\displaystyle  g(t) = e^{\lambda_0 t} \tilde{g} \left( \frac{1-e^{-\lambda_0 t}}{2\lambda_0} \right)=e^{\lambda_0 t}\tilde{g}(s),$$
  which implies
  \begin{align*}
    B_{g(t)} \left(x, \sqrt{t} \right) = B_{e^{\lambda_0 t} \tilde{g}(s)} \left(x, \sqrt{t} \right)=B_{\tilde{g}(s)} \left(x, \sqrt{e^{-\lambda_0 t} t} \right).
  \end{align*}
  If $t \in (0, 2\delta_0]$, then $s \in \left(0, \frac{4}{3} \delta_0 \right]$. Note that $\tilde{g}(0)=g(0)$.
  Therefore, Proposition~\ref{prn:pseudo-locality_A21} can be applied to obtain the following estimates.
  \begin{align*}
    \begin{cases}
    &\left|B_{g(t)} \left(x, \sqrt{t} \right) \right|_{d\mu_{g(t)}} = e^{\frac{m\lambda_0 t}{2}} \left| B_{\tilde{g}(s)}\left(x, \sqrt{e^{-\lambda_0 t}t} \right)\right|_{d\mu_{\tilde{g}(s)}}
      >\frac{1}{2} \left| B_{\tilde{g}(s)}\left(x, \sqrt{s} \right)\right|_{d\mu_{\tilde{g}(s)}}> \frac{1}{2} \kappa' s^{\frac{m}{2}} \triangleq \kappa_0 t^{\frac{m}{2}},\\
    &|Rm|(x, t)= e^{-\lambda_0 t} \left|\widetilde{Rm} \right|(x, s) \leq \frac{e^{-\lambda_0 t}}{100} s^{-1} < \frac{3}{100}e^{-\lambda_0 t} t^{-1} <t^{-1},
  \end{cases}
  \end{align*}
  for every point $x \in \Omega'=B_{g(0)} \left(x_0, \frac{3}{4} \right), \; t \in (0, 2\delta_0]$. So we finish the proof of Theorem~\ref{thm:pseudo-locality_A22}.
\end{proof}

\section{Curvature, distance and volume estimates}

Under the Ricci flow, evolution of distance between two points is controlled by the Ricci curvature.
By maximum principle, a scalar-flat Ricci flow solution must be Ricci flat. Therefore, the distance between
any two points does not depend on the time.  In this section, we will develop an ``almost"-version of
this observation. Fix two points in the underlying manifold of a normalized Ricci flow solution.
If the normalized scalar curvature is almost zero in the $L^1$-sense, then the distance between these
two points are almost fixed by the flow.  This new estimate is based on Proposition~\ref{prn:dgoup},
Theorem~\ref{thm:pseudo-locality_A22}, and the following estimate of normalized Ricci curvature.

\begin{lemma}
  Suppose $\left\{ (X,x_0,g(t)), -2 \leq t \leq 1 \right\}$ satisfies the following conditions.
  \begin{itemize}
    \item $g(t)$ satisfies the normalized Ricci flow solution
\begin{align*}
  \D{}{t} g_{ij} = -R_{ij} + \lambda_0 g_{ij}
\end{align*}
where $\lambda_0$ is a constant with $|\lambda_0|\leq \frac{1}{100m^2}$.
\footnote{Note that this is not $1$. In our mind, the flow in this lemma comes from the blowup of a general normalized flow, so the coefficient $\lambda_0$ could be very small.}

    \item $|Rm|(x, t) \leq \frac{1}{100m^2}$ whenever $x \in B_{g(t)}(x_0, 100), \; t \in [-2, 1]$.
    \item $inj(x_0, t) \geq 100$ uniformly for every $t \in [-2,1]$.
  \end{itemize}
  Then there exists a large constant $C=C(m)$ such that
  \begin{align}
    |Ric-\lambda_0 g|(x_0, 0)
    \leq C \left\{ \int_{-2}^{1} \int_{B_{g(0)}(x_0, 10)} |R-m\lambda_0|d\mu dt \right\}^{\frac{1}{2}}.
    \label{eqn:ricrhalf_normalized}
  \end{align}
  \label{lma:lcRm_nm}
\end{lemma}

\begin{proof}

For simplicity of notation, we denote $Ric-\lambda_0 g$ by $h$,  denote $R-m\lambda_0$ by $H$.

Recall that $|h|$ satisfies inequality (\ref{eqn:normricupper}). Locally,  $|Rm|$ is uniformly bounded.
So we should be able to control the $L^{\infty}$-norm of $|h|$ by the $L^2$-norm of $|h|$.
Actually, define
$\Omega=B_{g(0)}(x_0, 1), \; \Omega'=B_{g(0)} \left(x_0, \frac{1}{2} \right)$, $D=\Omega \times [-1, 0]$, $D'=\Omega' \times [-\frac{1}{2}, 0].$
By the second and the third condition, we obtain that $(\Omega, g(t))$ has a uniform Sobolev constant $\sigma=\sigma(m)$.
   Similar to Theorem 3.2 of~\cite{BWa2},  Moser iteration for the term $h=Ric-\lambda_0 g$ implies
\begin{align}
  \sup_{D'} |h| \leq C(m) \left \{\iint_D |h|^2 d\mu dt \right\}^{\frac{1}{2}}.
  \label{eqn:ricbdl2}
\end{align}

Choose cutoff function $\tilde{\eta}(y,t)=\psi(d_{g(t)}(y, x_0)-2)$, where $\psi$ is a smooth function which achieves value $1$ on $(-\infty, 0]$ and $0$ on
$[1, \infty)$, which also satisfies $|\psi'| \leq 2$.  Recall that $|\lambda_0| \leq \frac{1}{100m^2}$. So we have
\begin{align*}
  h(V, V) \leq \left( \frac{m-1}{100m^2} + |\lambda_0|  \right) g(V, V) \leq \frac{1}{100m} g(V, V)
\end{align*}
whenever $V \in TX$ and $|Rm|(V,V) \leq \frac{1}{100m^2}g(V,V)$. By the evolution of geodesic length, it is easy to check that
\begin{align*}
 &\Omega=B_{g(0)}(x_0, 1) \subset B_{g(t)}(x_0, 2),\\
 & B_{g(t)}(x_0, 3) \subset W=B_{g(0)}(x_0, 10),
\end{align*}
for every $-2 \leq t \leq 1$. Therefore $\tilde{\eta} \equiv 1$ on $\Omega$, $\tilde{\eta} \equiv 0$ outside $W$ whenever $-2 \leq t \leq 1$.

By mean value theorem of calculus, we can assume $t_1, t_2$ satisfies the following properties.
\begin{align}
  &-2 \leq t_1 \leq -1, &\quad \left. \int_W |H| d\mu \right|_{t_1}  \leq  \int_{-2}^{-1} \int_W |H|d\mu dt \leq \int_{-2}^{1} \int_W |H|d\mu dt. \label{eqn:mv_1}\\
  &0 \leq t_2 \leq 1, &\quad \left. \int_W |H| d\mu \right|_{t_2}  \leq  \int_{0}^{1} \int_W |H|d\mu dt \leq \int_{-2}^{1} \int_W |H|d\mu dt. \label{eqn:mv_2}
\end{align}

Using the evolution equation of normalized scalar curvature equation
(\ref{eqn:normscalar}), similar to the calculation in~\cite{BWa2}, we obtain that
\begin{align}
  &\quad \int_{t_1}^{t_2} \int_{\Omega} |h|^2 d\mu dt  \nonumber \\
  &\leq \int_{t_1}^{t_2} \int_X \tilde{\eta} |h|^2 d\mu dt \nonumber\\
  &=\int_{t_1}^{t_2} \int_X  \tilde{\eta} \left( \D{H}{t} -\frac{1}{2} \Delta H -\lambda_0 H \right) d\mu dt \nonumber\\
  &= \left. \left( \int_X \tilde{\eta} H d\mu \right) \right|_{t_1}^{t_2} - \int_{t_1}^{t_2} \int_X H \left( \D{}{t} \tilde{\eta}
  +\frac{1}{2} \Delta \tilde{\eta}+\left(\lambda_0 -\frac{H}{2} \right) \tilde{\eta} \right) d\mu dt \nonumber\\
  &\leq C\left\{ \left.  \int_W |H| d\mu \right|_{t=t_2} + \left.  \int_W |H| d\mu \right|_{t=t_1} + \int_{t_1}^{t_2} \int_W |H| d\mu dt \right\}.
  \label{eqn:ricrl1}
\end{align}
Note that $[-1, 0] \subset [t_1, t_2] \subset [-2, 1]$.
Combining (\ref{eqn:ricbdl2}), (\ref{eqn:mv_1}), (\ref{eqn:mv_2}) and (\ref{eqn:ricrl1}) yields
\begin{align}
  \sup_{D'} |h| \leq C \left \{\iint_D |h|^2 d\mu dt \right\}^{\frac{1}{2}}
  \leq C\left\{\int_{-2}^{1} \int_W |H| d\mu dt \right\}^{\frac{1}{2}},   \label{eqn:ricrl1_2}
\end{align}
where $C$ depends only on the dimension $m$.
\end{proof}

Combine Lemma~\ref{lma:lcRm_nm} and Theorem~\ref{thm:pseudo-locality_A22}, we have the following estimate.

\begin{lemma}
Suppose $\left\{ (X, x_0, g(t)), 0 \leq t \leq 1 \right\}$ satisfies all the conditions in Theorem~\ref{thm:pseudo-locality_A22}. Then
  \begin{align}
    |Ric-\lambda_0 g|(x, s) \leq C(m) s^{-\frac{m+4}{2}} \left\{ \int_0^{2s} \int_{\Omega} |R-m\lambda_0| d\mu dt \right\}^{\frac{1}{2}},
    \quad \forall x \in \Omega'=B_{g(0)} \left(x_0, \frac12\right), \;
     s \in \left(0, \delta_0 \right].
  \label{eqn:f20_ricup}
  \end{align}
  \label{lma:Ricbyt}
\end{lemma}

\begin{proof}
   By Theorem~\ref{thm:pseudo-locality_A22}, we have
   \begin{align}
     |Rm|(x, t) \leq t^{-1}, \quad \left|B_{g(t)} \left(x, \sqrt{t} \right)\right|_{d\mu_{g(t)}} \geq \kappa \left(\sqrt{t} \right)^m,
     \label{eqn:cv}
   \end{align}
   for every point $y \in B_{g(0)} \left( x_0, \frac{3}{4} \right)$, $t \in (0, 2\delta_0]$.

   Fix $x \in \Omega'=B_{g(0)} \left(x_0, \frac12\right), \; s \in (0, \delta_0]$.  By (\ref{eqn:cv}), the injectivity radius estimate in~\cite{CGT} yields that
   \begin{align}
     inj(x, t) \geq \xi \sqrt{s},    \label{eqn:injlow}
   \end{align}
   for some constant $\xi=\xi(m, \kappa(m))=\xi(m)$ whenever $\frac{s}{2} \leq t \leq 2s$.
Put
$$A=1000m\xi^{-1} s^{-\frac{1}{2}}, \quad \tilde{g}(t)=A^2 g\left(A^{-2}t +s \right).$$
Clearly, $\tilde{g}$ satisfies the evolution equation
$$\displaystyle \D{}{t} \tilde{g}= -\widetilde{Ric} + A^{-2}\lambda_0 \tilde{g}.$$
   In view of (\ref{eqn:cv}) and (\ref{eqn:injlow}), we have injectivity radius estimate and curvature estimate required by Lemma~\ref{lma:lcRm_nm}. It follows that
   \begin{align*}
     \left|\widetilde{Ric}-A^{-2} \lambda_0 \tilde{g} \right|^2 (x, 0) \leq C \int_{-1}^2 \int_{B_{\tilde{g}_0}(x, 10)} \left|\tilde{R}-mA^{-2}\lambda_0 \right|d \tilde{\mu} dt,
   \end{align*}
  which is the same as the following inequality before scaling:
  \begin{align}
    |Ric-\lambda_0 g|^2(x, s) \leq CA^{m+4} \int_{s-A^{-2}}^{s+2A^{-2}} \int_{B_{g(s)}\left(x, 10A^{-1}\right)} |R-m\lambda_0| d\mu dt.
  \label{eqn:ricbeforescale}
  \end{align}
  Recall that in the definition $A=1000m\xi^{-1} s^{-\frac{1}{2}}$, $1000m\xi^{-1}$ is a constant depending only on $m$. Therefore, (\ref{eqn:ricbeforescale}) implies
  \begin{align}
    |Ric-\lambda_0 g|(x, s) &\leq C(m)s^{-\frac{m+4}{2}}  \left\{ \int_{s-A^{-2}}^{s+2A^{-2}}
    \int_{B_{g(s)}\left(x, 10A^{-1} \right)} |R-m\lambda_0| d\mu dt \right\}^{\frac{1}{2}} \label{eqn:f20_riclocal}.
  \end{align}
  By inequality (\ref{eqn:dcompare_step1}), whose proof is independent,  we obtain that
  \begin{align}
    B_{g(s)} \left(x, 10A^{-1} \right) \subset B_{g(s)} \left(x, \frac{1}{8}-C\sqrt{s} \right) \subset B_{g(0)} \left(x, \frac{1}{8} \right)
    \subset B_{g(0)} \left(x_0, \frac{3}{4} \right) \subset \Omega=B_{g(0)}\left( x_0, 1 \right).
    \label{eqn:f20_precontaining}
  \end{align}
  Then inequality (\ref{eqn:f20_ricup}) follows from (\ref{eqn:f20_riclocal}), (\ref{eqn:f20_precontaining}) and the fact that
  $[s-A^{-2}, s+2A^{-2}] \subset [0, 2s]$.
\end{proof}

Recall Proposition~\ref{prn:dgoup}, estimate (\ref{eqn:f20_ricup}) implies that distance is almost expanding along the flow.
\begin{lemma}
  Suppose  $\left\{ (X, x_0, g(t)), 0 \leq t \leq 1 \right\}$ satisfies all the conditions in Theorem~\ref{thm:pseudo-locality_A22}.
  Then for every time $t_0 \in (0, \delta_0]$ and every two points $x_1, x_2 \in \Omega'=B_{g(0)} \left(x_0, \frac{1}{2} \right)$, we have
   \begin{align}
     &d_{g(t_0)}(x_1, x_2) \geq  d_{g(0)}(x_1, x_2) - C\sqrt{t_0}, \label{eqn:dcompare_step1} \\
     &d_{g(\delta_0)}(x_1, x_2) \geq d_{g(0)}(x_1, x_2) - C\left( \sqrt{t_0} + t_0^{-\frac{m+2}{2}}E^{\frac{1}{2}} \right), \label{eqn:dcompare_step2}
   \end{align}
  where $C=C(m)$ is a universal constant, $E=\int_0^{2 \delta_0} \int_{\Omega} |R-m\lambda_0| d\mu dt$.   In particular, if
  $E < \delta_0^{m+3}$, then we have
  \begin{align}
    d_{g(\delta_0)}(x_1, x_2) \geq d_{g(0)}(x_1, x_2) - CE^{\frac{1}{2(m+3)}}. \label{eqn:dcompare_step3}
  \end{align}
 \label{lma:dcompare}
\end{lemma}

\begin{proof}
Let us first prove inequality (\ref{eqn:dcompare_step1}).

By inequality (\ref{eqn:dgoup}) and inequality (\ref{eqn:pseudo}), we have
\begin{align*}
  \frac{d}{dt} d_{g(t)}(x_1, x_2) \geq \frac{1}{2} \lambda_0 d_{g(0)}(x_1, x_2) - Ct^{-\frac{1}{2}},  \quad \forall t \in (0, t_0],
\end{align*}
where $C$ is a universal constant. Consequently, we have
\begin{align*}
  &\qquad \frac{d}{dt} \left(  e^{-\frac{\lambda_0 t}{2} } d_{g(t)}(x_1, x_2) \right) \geq -Ce^{-\frac{\lambda_0 t}{2}} t^{-\frac{1}{2}} \geq -C t^{-\frac{1}{2}}, \\
  &\Rightarrow   e^{-\frac{\lambda_0 t_0}{2}} d_{g(t_0)}(x_1, x_2) - d_{g(0)}(x_1, x_2) \geq -C\sqrt{t_0},\\
  &\Rightarrow   d_{g(t_0)}(x_1, x_2) \geq e^{\frac{\lambda_0 t_0}{2}} \left( d_{g(0)}(x_1, x_2) -C\sqrt{t_0} \right).
\end{align*}
If $\lambda_0 \geq 0$, we have already obtain inequality (\ref{eqn:dcompare_step1}) trivially.
If $\lambda_0<0$, we have
\begin{align*}
   d_{g(t_0)}(x_1, x_2) &\geq  \left( d_{g(0)}(x_1, x_2) -C\sqrt{t_0} \right) + \left(e^{\frac{\lambda_0 t_0}{2}}-1 \right) \left( d_{g(0)}(x_1, x_2) -C\sqrt{t_0} \right)\\
      &\geq  \left( d_{g(0)}(x_1, x_2) -C\sqrt{t_0} \right)  -Ct_0 \left|d_{g(0)}(x_1, x_2) -C\sqrt{t_0} \right|\\
      &\geq  \left( d_{g(0)}(x_1, x_2) -C\sqrt{t_0} \right)  -Ct_0\\
      &\geq  d_{g(0)}(x_1, x_2) -C\sqrt{t_0}.
\end{align*}
So we finish the proof of inequality (\ref{eqn:dcompare_step1}).

We continue to prove inequality (\ref{eqn:dcompare_step2}).
Along the normalized Ricci flow, the derivative of logarithm of geodesic length is bounded by the term $|Ric-\lambda_0 g|$ on the geodesic.
Therefore, estimate (\ref{eqn:f20_ricup}) yields the following inequalities.
    \begin{align*}
     \left|\log \frac{d_{g(\delta_0)}(x_1, x_2)}{d_{g(t_0)}(x_1, x_2)} \right| \leq  C \int_{t_0}^{\delta_0} t^{-\frac{m+4}{2}} E^{\frac{1}{2}} dt
     \leq  CE^{\frac{1}{2}} \left( t_0^{-\frac{m+2}{2}} - \delta_0^{-\frac{m+2}{2}} \right) \leq C E^{\frac{1}{2}} t_0^{-\frac{m+2}{2}}.
    \end{align*}
It follows that
    \begin{align*}
      d_{g(\delta_0)}(x_1, x_2) &\geq  d_{g(t_0)}(x_1, x_2) e^{-Ct_0^{-\frac{m+2}{2}}E^{\frac{1}{2}}}\\
         &\geq \left( d_{g(0)}(x_1, x_2) - Ct_0^{\frac{1}{2}} \right) e^{-Ct_0^{-\frac{m+2}{2}}E^{\frac{1}{2}}}\\
         &=\left( d_{g(0)}(x_1, x_2) - Ct_0^{\frac{1}{2}} \right) + \left(e^{-Ct_0^{-\frac{m+2}{2}}E^{\frac{1}{2}}}-1 \right) \cdot \left( d_{g(0)}(x_1, x_2) - Ct_0^{\frac{1}{2}} \right)\\
	 &\geq \left( d_{g(0)}(x_1, x_2) - Ct_0^{\frac{1}{2}} \right) -Ct_0^{-\frac{m+2}{2}}E^{\frac{1}{2}} \left| d_{g(0)}(x_1, x_2) - Ct_0^{\frac{1}{2}} \right|\\
	 &\geq d_{g(0)}(x_1, x_2) - C \left( t_0^{\frac{1}{2}} + t_0^{-\frac{m+2}{2}}E^{\frac{1}{2}} \right).
    \end{align*}
So we finish the proof of inequality (\ref{eqn:dcompare_step2}).

If $E< \delta_0^{m+3}$, then $E^{\frac{1}{m+3}}<\delta_0$. Let $t_0=E^{\frac{1}{m+3}}$ and plug it into inequality (\ref{eqn:dcompare_step2}), we obtain
inequality (\ref{eqn:dcompare_step3}).
\end{proof}

\begin{corollary}
  Same conditions as in Lemma~\ref{lma:dcompare}.    If $E << \delta_0^{m+3}$, $x_1 \in \Omega'=B_{g(0)} \left(x_0, \frac{1}{2} \right)$,
then
\begin{align}
   B_{g(\delta_0)} \left(x_1, r-CE^{\frac{1}{2(m+3)}} \right) \subset B_{g(0)} \left(x_1, r\right),
\label{eqn:containing}
\end{align}
for every $0<r<\frac{1}{2}-d_{g(0)}(x_0, x_1)$.   In particular, we have
\begin{align}
   B_{g(\delta_0)} \left(x_0, r-CE^{\frac{1}{2(m+3)}}\right) \subset B_{g(0)} \left(x_0, r\right), \quad \forall \; 0<r<\frac12.
\label{eqn:containing_center}
\end{align}
\label{cly:containing}
\end{corollary}

\begin{proof}
 Direct application of inequality (\ref{eqn:dcompare_step3}).
\end{proof}

Intuitively, an almost expanding map which almost fix volume must be an almost isometry.  This observation can be achieved
precisely by Theorem~\ref{thm:isometry_almost}.  However, in order to obtain Theorem~\ref{thm:isometry_almost}, we first
need an estimate to prevent the distance to expand too fast, which is the meaning of the following Lemma.

\begin{lemma}
   Suppose $\left\{ (X, x_0, g(t)), 0 \leq t \leq 1 \right\}$ satisfies all the conditions in Theorem~\ref{thm:pseudo-locality_A22}.

   Let $\Omega=B_{g(0)}(x_0, 1), \Omega'=B_{g(0)} \left(x_0, \frac{1}{2} \right)$. For every $l<\frac{1}{2}$, define
\begin{align*}
  &A_{+,l}=\sup_{B_{g(0)}(x, r) \subset \Omega', 0<r\leq l} \omega_m^{-1}r^{-m} \left| B_{g(0)}(x, r)\right|_{d\mu_{g(0)}}, \\
  &A_{-,l}=\inf_{B_{g(\delta_0)}(x, r) \subset \Omega', 0<r\leq l} \omega_m^{-1}r^{-m} \left| B_{g(\delta_0)}(x, r)\right|_{d\mu_{g(\delta_0)}}.
\end{align*}

If $x_1, x_2 \in \Omega''=B_{g(0)}\left(x_0, \frac{1}{4}\right)$, $l=d_{g(0)}(x_1, x_2) < \frac{1}{8}$, then we have
  \begin{align}
    l- C E^{\frac{1}{2(m+3)}} \leq   d_{g(\delta_0)}(x_1, x_2) \leq
  l+CA_{+,4l} \left\{ \left| \frac{A_{+,l}}{A_{-,l}} -1 \right|^{\frac{1}{m}} +  l^{-\frac{1}{m}}E^{\frac{1}{2m(m+3)}}  \right\}l
    \label{eqn:metricequivalent}
  \end{align}
  whenever  $\displaystyle E=\int_0^{2 \delta_0} \int_{\Omega} |R-m\lambda_0| d\mu dt<<  l^{2(m+3)}$.
 \label{lma:metricequivalent}
\end{lemma}

\begin{proof}
  The left hand side of inequality (\ref{eqn:metricequivalent}) follows directly from inequality  (\ref{eqn:dcompare_step3}).  So we focus on
  the proof of the right hand side of inequality (\ref{eqn:metricequivalent}).

  We denote the constant in Lemma~\ref{lma:dcompare} by $C_0$ and fix it in this proof. All the other $C$'s may be different from line to line.

 Among all the geodesic balls in $B_{g(0)}(x_1, l)$, let $B_{g(0)}(x, r_0)$ be the largest geodesic ball (counted by radius under $g(0)$) such that
$$\displaystyle B_{g(0)}(x, r_0) \cap B_{g(\delta_0)}\left(x_1,  l- C_0 E^{\frac{1}{2(m+3)}}  \right) =\emptyset.$$
See Figure~\ref{fig:balls} for intuition.

\begin{figure}[h]
 \begin{center}
    \psfrag{A}[c][c]{$x_1$}
    \psfrag{B}[c][c]{$x_2$}
    \psfrag{C}[c][c]{$x_3$}
    \psfrag{D}[c][c]{$x$}
    \psfrag{B1}[l][l]{$red \; ball=B_{g(0)}(x_1, l)$}
    \psfrag{B2}[l][l]{$blue \; ball=B_{g(\delta_0)}(x_1, l-C_0E^{\frac{1}{2(m+3)}})$}
    \psfrag{B3}[l][l]{$green \; ball=B_{g(0)}(x_2, 3r_0)$}
    \psfrag{B4}[l][l]{$yellow \; ball=B_{g(0)}(x, r_0)$}
    \psfrag{t0}[c][c]{$t=0$}
    \psfrag{t1}[c][c]{$t=\delta_0$}
    \psfrag{t}[c][c]{$t$}
    \includegraphics[width=0.8\columnwidth]{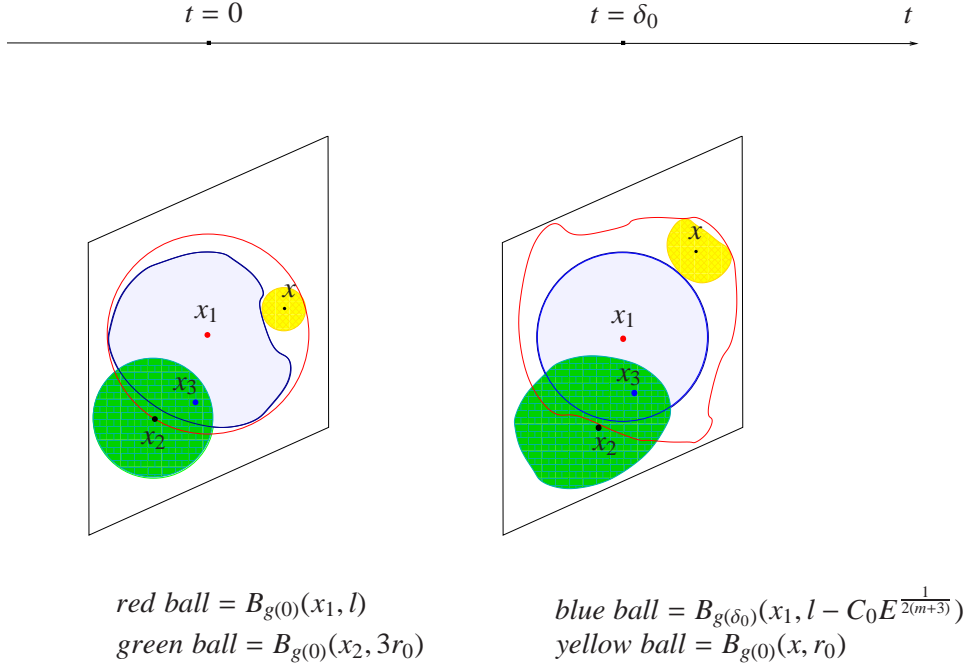}
  \caption{The relationship among the balls}
  \label{fig:balls}
 \end{center}
\end{figure}

 \begin{claim}
   The radius $r_0$ is bounded from above by the following inequality.
   \begin{align}
    r_0 \leq   \left\{ \left| \frac{A_{+,l}}{A_{-,l}} -1 \right| + Cl^{-1} E^{\frac{1}{2(m+3)}} \right\}^{\frac{1}{m}} l + C_0 E^{\frac{1}{2(m+3)}}.
   \label{eqn:extrarup}
\end{align}
\label{clm:f19_rupper}
 \end{claim}

 By definition,  $B_{g(0)}(x, r_0)$ and the ball $B_{g(\delta_0)}\left(x_1,  l- C_0 E^{\frac{1}{2(m+3)}}  \right)$ are disjoint. Moreover, Corollary~\ref{cly:containing} implies
 that $\displaystyle B_{g(0)}(x, r_0) \cup B_{g(\delta_0)} \left(x_1,  l- C_0 E^{\frac{1}{2(m+3)}}  \right) \subset B_{g(0)}(x_1, l)$.  Therefore, we have
\begin{align}
  \left| B_{g(0)}(x, r_0) \right|_{d\mu_{g(\delta_0)}} &\leq \left| B_{g(0)} (x_1, l)\right|_{d\mu_{g(\delta_0)}}
  - \left| B_{g(\delta_0)} \left(x_1,  l- C_0 E^{\frac{1}{2(m+3)}} \right)\right|_{d\mu_{g(\delta_0)}}  \nonumber\\
 &\leq  \left| B_{g(0)}(x_1, l)\right|_{d\mu_{g(0)}}
  - \left| B_{g(\delta_0)} \left(x_1,  l- C_0 E^{\frac{1}{2(m+3)}} \right)\right|_{d\mu_{g(\delta_0)}} +E.
  \label{eqn:f19_rup}
\end{align}
By Corollary~\ref{cly:containing}, we have  $\displaystyle B_{g(\delta_0)} \left(x, r_0- C_0 E^{\frac{1}{2(m+3)}}  \right) \subset B_{g(0)}(x, r_0)$. Note that $r_0<l$ by definition. It follows
from the definition of $A_{-,l}$ that
\begin{align}
  \left| B_{g(0)}(x, r_0) \right|_{d\mu_{g(\delta_0)}} \geq \left| B_{g(\delta_0)} \left(x, r_0- C_0 E^{\frac{1}{2(m+3)}}  \right) \right|_{d\mu_{g(\delta_0)}}
  \geq A_{-,l} \left( r_0- C_0 E^{\frac{1}{2(m+3)}}  \right)^m.
  \label{eqn:f19_rlow}
\end{align}
Plugging (\ref{eqn:f19_rlow}) into (\ref{eqn:f19_rup}) yields
\begin{align}
  A_{-,l} \left( r_0- C_0 E^{\frac{1}{2(m+3)}}  \right)^m
  &\leq \left| B_{g(0)}(x_1, l)\right|_{d\mu_{g(0)}} - \left| B_{g(\delta_0)} \left(x_1,  l- C_0E^{\frac{1}{2(m+3)}} \right)\right|_{d\mu_{g(\delta_0)}} +E \nonumber\\
  &\leq A_{+,l}l^m -A_{-,l} l^m \left( 1- C_0 l^{-1}E^{\frac{1}{2(m+3)}}\right)^m  +E \nonumber\\
  &\leq A_{+,l}l^m -A_{-,l} l^m \left( 1-2mC_0 l^{-1} E^{\frac{1}{2(m+3)}}\right) +E \nonumber\\
  &\leq l^m \left\{ \left( A_{+,l}-A_{-,l} \right)  + 2mC_0A_{-,l} l^{-1} E^{\frac{1}{2(m+3)}} + l^{-m}E \right\},
  \label{eqn:f19_rup_2}
\end{align}
where we use the fact that $\displaystyle C_0 l^{-1}E^{\frac{1}{2(m+3)}} << 1$ in the third step, last step respectively.
By the non-collapsed condition at time $t=\delta_0$, we obtain that $A_{-, l}\geq C(m,\kappa(m))=c(m)$. By the definition of $A_{-,l}$, we automatically have
$A_{-,l} \leq 1$.  Note also that $\displaystyle l^{-m}E << l^{-1} E^{\frac{1}{2(m+3)}}$. We obtain
\begin{align}
    2mC_0A_{-,l} l^{-1} E^{\frac{1}{2(m+3)}} + l^{-m}E < C A_{-,l} l^{-1} E^{\frac{1}{2(m+3)}}.
    \label{eqn:f19_bridge}
\end{align}
Combining (\ref{eqn:f19_bridge}) and (\ref{eqn:f19_rup_2}) yields
\begin{align}
  r_0 \leq   \left\{ \left| \frac{A_{+,l}}{A_{-,l}} -1 \right| + Cl^{-1} E^{\frac{1}{2(m+3)}} \right\}^{\frac{1}{m}} l + C_0 E^{\frac{1}{2(m+3)}}.
  \label{eqn:f19_rup_3}
\end{align}

Note that there is a point $x_3 \in B_{g(0)}(x_2, 3r_0)$ such that $x_3 \in B_{g(\delta_0)} \left(x_1, l- C_0E^{\frac{1}{2(m+3)}}  \right)$.
Otherwise,  let $\alpha$ be a unit speed geodesic (under metric $g(0)$) connecting $x_1$ and $x_2$ such that $\alpha(0)=x_1$, $\alpha(l)=x_2$.
By triangle inequality, we can see that
\begin{align*}
  B_{g(0)} \left( \alpha\left( l-\frac{3}{2} r_0 \right),  \frac{5}{4}r_0 \right) \cap B_{g(\delta_0)} \left(x_1, l- C_0E^{\frac{1}{2(m+3)}}  \right)
  \subset B_{g(0)}(x_2, 3r_0) \cap B_{g(\delta_0)} \left(x_1, l- C_0E^{\frac{1}{2(m+3)}}  \right) =\emptyset,
\end{align*}
which contradicts to the definition of $r_0$.

\begin{claim}
There exists a constant $C=C(m)$ such that
  \begin{align}
    d_{g(\delta_0)}(x_2, x_3) \leq  C A_{+,4l} \max\left\{3C_0 E^{\frac{1}{2(m+3)}}, r_0 \right\}.
  \label{eqn:y12distance}
\end{align}
  \label{clm:f18_ball}
\end{claim}

We first consider the case that $r_0>3C_0 E^{\frac{1}{2(m+3)}}$.

  Under metric $g(0)$, let $\gamma$ be a shortest geodesic connecting $x_2, x_3$. Clearly, $|\gamma|_{g(0)}\leq 3r_0$.
  Under the metric $g(\delta_0)$, $\gamma$ may not be a shortest geodesic. However, it is still a smooth curve.
  Cover the curve $\gamma$ by geodesic balls $B_{g(\delta_0)}(z_i, r_0)$ with the following properties.
  \begin{itemize}
    \item $\displaystyle z_i \in \gamma, \quad \forall \; i \in \left\{ 1, \cdots, N \right\}$;
    \item $\gamma \subset \bigcup_{i=1}^N B_{g(\delta_0)}(z_i, r_0)$;
    \item $B_{g(\delta_0)} \left(z_i, \frac{r_0}{2} \right)$ are disjoint.
  \end{itemize}
  Since  $z_i \in \gamma$, we have $z_i \in B_{g(0)}(x_2, 3r_0) \subset B_{g(0)} \left(x_0, \frac{1}{2} \right)$.
  Note that $r_0>3C_0E^{\frac{1}{2(m+3)}}$, Corollary~\ref{cly:containing} implies that
  \begin{align}
   & B_{g(\delta_0)} \left(z_i, \frac{r_0}{2} \right)
    \subset  B_{g(0)}\left(z_i, \frac{r_0}{2} + C_0E^{\frac{1}{2(m+3)}} \right) \subset B_{g(0)}(z_i, r_0) \subset B_{g(0)}(x_2, 4r_0) \nonumber\\
   &\quad \subset B_{g(0)} \left(x_0, \frac{1}{4}+4r_0 \right) \subset \Omega'=B_{g(0)}\left( x_0, \frac{1}{2} \right)
   \subset \Omega=B_{g(0)}(x_0, 1).
   \label{eqn:containsequence}
  \end{align}
  Note that $B_{g(\delta_0)} \left(z_i, \frac{r_0}{2}\right)$ are disjoint, we obtain
  \begin{align}
    \sum_{i=1}^N \left|B_{g(\delta_0)} \left(z_i, \frac{r_0}{2} \right) \right|_{d\mu_{g(0)}} \leq \left| B_{g(0)}(x_2, 4r_0) \right|_{d\mu_{g(0)}}.
    \label{eqn:f19_vup_1}
  \end{align}
  By the evolution equation of volume form and (\ref{eqn:containsequence}), we have
  \begin{align}
    \left| \sum_{i=1}^N \left|B_{g(\delta_0)} \left(z_i, \frac{r_0}{2} \right) \right|_{d\mu_{g(0)}}
     -\sum_{i=1}^N \left|B_{g(\delta_0)} \left(z_i, \frac{r_0}{2} \right) \right|_{d\mu_{g(\delta_0)}} \right| <E.
     \label{eqn:f19_vup_2}
  \end{align}
  It follows from (\ref{eqn:f19_vup_1}) and (\ref{eqn:f19_vup_2}) that
  \begin{align}
    \sum_{i=1}^N \left|B_{g(\delta_0)} \left(z_i, \frac{r_0}{2} \right) \right|_{d\mu_{g(\delta_0)}} \leq \left| B_{g(0)}(x_2, 4r_0) \right|_{d\mu_{g(0)}} + E.
    \label{eqn:f19_vup_3}
  \end{align}
  Since $r_0<l$,  the definition of $A_{-,l}$, $A_{+,l}$ implies the following inequalities.
  \begin{align}
    &\left|B_{g(\delta_0)} \left(z_i, \frac{r_0}{2} \right) \right|_{d\mu_{g(\delta_0)}} \geq A_{-,l} \left( \frac{r_0}{2} \right)^m,
    \quad \forall \; i \in \left\{ 1, \dots, N \right\};  \label{eqn:f19_vup_4}\\
    &\left| B_{g(0)}(x_1, 4r_0) \right|_{d\mu_{g(0)}} \leq A_{+,4l} (4r_0)^m. \label{eqn:f19_vup_5}
  \end{align}
  Combine (\ref{eqn:f19_vup_3}), (\ref{eqn:f19_vup_4}) and (\ref{eqn:f19_vup_5}), we obtain
  \begin{align}
    \frac{NA_{-,l}}{2^m} r_0^m \leq 4^{m} A_{+,4l}r_0^m + E, \Rightarrow N \leq
    2^m \left( 4^m A_{+,4l} +Er_0^{-m} \right)A_{-,l}^{-1}.
    \label{eqn:f19_nup_1}
  \end{align}
  Recall that $\displaystyle \bigcup_{i=1}^N B_{g_i(\delta_0)}(z_i, r_0)$ is a covering of $\gamma$. Therefore, (\ref{eqn:f19_nup_1}) implies
  \begin{align}
    d_{g(\delta_0)}(x_2, x_3) \leq 2Nr_0 \leq  2^{m+1} \left( 4^m A_{+,4l} + Er_0^{-m} \right) A_{-,l}^{-1}r_0.
    \label{eqn:f19_dup_1}
  \end{align}
  On one hand, by (\ref{eqn:pseduo_noncollapsed}), non-collapsed condition at time $t=\delta_0$ implies that $A_{-,l}^{-1}$ is bounded from above uniformly.
  On the other hand, $A_{+,4l}$ is bounded from below in view of the volume comparison and (\ref{eqn:pseudocon}).
  Therefore, the fact that $r_0>3C_0E^{\frac{1}{2(m+3)}}$ implies   $\displaystyle Er_0^{-m}<(3C_0)^{-m}E^{\frac{m+6}{2(m+3)}}<CA_{+,4l}$.
  Consequently, we can simplify (\ref{eqn:f19_dup_1}) to
\begin{align*}
  d_{g(\delta_0)}(x_2, x_3) \leq  C A_{+,4l}r_0,
\end{align*}
which is the same as (\ref{eqn:y12distance}) under our assumption $r>3C_0 E^{\frac{1}{2(m+3)}}$.
If $r_0 \leq 3C_0 E^{\frac{1}{2(m+3)}}$, we can repeat the previous argument by choosing covering balls of radius $3C_0 E^{\frac{1}{2(m+3)}}$.
The details are similar, so we omit them. \\

Now we can combine Claim~\ref{clm:f19_rupper}  and Claim~\ref{clm:f18_ball} to obtain precise upper bound of $d_{g(\delta_0)}(x_2, x_3)$.
If $r_0 \leq 3C_0E^{\frac{1}{2(m+3)}}$, we obtain
\begin{align}
  d_{g(\delta_0)}(x_2, x_3) \leq CA_{+,4l} E^{\frac{1}{2(m+3)}}< CA_{+, 4l} l^{\frac{m-1}{m}}E^{\frac{1}{2m(m+3)}}
  \label{eqn:f21_r0small}
\end{align}
since $E<< l^{2(m+3)}$. If $r_0>3C_0E^{\frac{1}{2(m+3)}}$, then we have
\begin{align}
  d_{g(\delta_0)}(x_2, x_3)
 &\leq CA_{+,4l}  \left\{ \left\{ \left| \frac{A_{+,l}}{A_{-,l}} -1 \right| + Cl^{-1} E^{\frac{1}{2(m+3)}} \right\}^{\frac{1}{m}} l + C_0 E^{\frac{1}{2(m+3)}} \right\} \nonumber\\
 &\leq CA_{+,4l} \left\{ \left| \frac{A_{+,l}}{A_{-,l}} -1 \right|^{\frac{1}{m}} l + C^{\frac{1}{m}} l^{\frac{m-1}{m}} E^{\frac{1}{2m(m+3)}} +C_0 E^{\frac{1}{2(m+3)}} \right\} \nonumber\\
 &\leq CA_{+,4l} \left\{ \left| \frac{A_{+,l}}{A_{-,l}} -1 \right|^{\frac{1}{m}}  + l^{-\frac{1}{m}} E^{\frac{1}{2m(m+3)}}  \right\} l.
\label{eqn:y12distance_2}
\end{align}
Therefore, triangle inequality yields that
\begin{align}
    d_{g(\delta_0)}(x_1, x_2)
    & \leq d_{g(\delta_0)}(x_1, x_3) + d_{g(\delta_0)}(x_3, x_2) \nonumber\\
    & \leq l-C_0E^{\frac{1}{2(m+3)}} + d_{g(\delta_0)}(x_3, x_2) \nonumber\\
    & < l+CA_{+,4l} \left\{ \left| \frac{A_{+,l}}{A_{-,l}} -1 \right|^{\frac{1}{m}} + l^{-\frac{1}{m}} E^{\frac{1}{2m(m+3)}}  \right\} l.
    \label{eqn:x12distance}
\end{align}
\end{proof}

By refining the estimate in Lemma~\ref{lma:metricequivalent}, we are able to prove that the distance is almost fixed whenever
the normalized scalar curvature is almost zero.

\begin{theorem}
 Suppose $\left\{ (X, x_0, g(t)), 0 \leq t \leq 1 \right\}$ satisfies all the conditions in Theorem~\ref{thm:pseudo-locality_A22}.
 Then for every two points $x_1, x_2 \in \Omega''=B_{g(0)} \left(x_0, \frac{1}{4} \right)$, $l=d_{g(0)}(x_1, x_2)$, we have
  \begin{align}
    l-CE^{\frac{1}{2(m+3)}} \leq d_{g(\delta_0)}(x_1, x_2) \leq l + ClE^{\frac{1}{3m(m+3)}}
  \label{eqn:f20_d2sides}
  \end{align}
  whenever $\displaystyle E= \int_0^{2 \delta_0} \int_{\Omega} |R-m\lambda_0| d\mu dt << l^{6(m+3)}$.
  Here $C=C(m, \delta_0(m))=C(m)$.
 \label{thm:isometry_almost}
\end{theorem}

\begin{proof}
   The first inequality of (\ref{eqn:f20_d2sides}) is the same as the one in (\ref{eqn:metricequivalent}).
   So we only need to show the second inequality of (\ref{eqn:f20_d2sides}).

    At time $t=\delta_0$, $|Rm|$ is uniformly bounded, injectivity radius is uniformly bounded from below.  Therefore, Rauch comparison theorem
    can be applied to obtain a lower bound of $A_{-,r}$.    At time $t=0$, Ricci curvature is bounded from below. So the Bishop volume
    comparison theorem implies an upper bound of $A_{+,r}$.  In short, we have
   \begin{align*}
     A_{+,r} \leq 1+ Cr^2, \quad A_{-,r} \geq 1-Cr^2,
   \end{align*}
   whenever $r<\xi=\xi(m, \kappa(m), \delta_0(m))=\xi(m)$.  It follows that
   \begin{align}
     CA_{+,4r} \left\{ \left| \frac{A_{+,r}}{A_{-,r}}-1 \right|^{\frac{1}{m}} + r^{-\frac{1}{m}}E^{\frac{1}{2m(m+3)}}  \right\}
     \leq C\left\{ r^{\frac{2}{m}} + r^{-\frac{1}{m}}E^{\frac{1}{2m(m+3)}}  \right\}.
   \label{eqn:f20_dup_r}
   \end{align}
   By (\ref{eqn:metricequivalent}) and (\ref{eqn:f20_dup_r}),  we have
   \begin{align}
     d_{g(\delta_0)}(y_1, y_2) r^{-1} \leq 1+ C\left\{ r^{\frac{2}{m}} + r^{-\frac{1}{m}}E^{\frac{1}{2m(m+3)}}  \right\},
   \label{eqn:f20_dup}
   \end{align}
   whenever $y_1, y_2 \in B_{g(0)} \left(x_0, \frac{1}{4} \right)$ and $d_{g(0)}(y_1, y_2)=r<\xi$.

   Fix a big integer number $N>\xi^{-1} l$.   Let $\gamma$ be a unit speed shortest geodesic connecting  $x_1, x_2$ such that $\gamma(0)=x_1$, $\gamma(l)=x_2$.
   Define $\displaystyle z_i=\gamma\left( N^{-1}il\right)$. Clearly, $z_0=x_1, \; z_N=x_2$.
   Since $\displaystyle d_{g(\delta_0)}(z_i, z_{i+1})=N^{-1}l<\xi$ for every $i=0, \cdots, N-1$, it follows from (\ref{eqn:f20_dup}) that
   \begin{align*}
     \frac{d_{g(\delta_0)} \left(z_i, z_{i+1} \right)}{N^{-1}l}
     \leq 1 + C\left\{ N^{-\frac{2}{m}} l^{\frac{2}{m}} + N^{\frac{1}{m}} l^{-\frac{1}{m}} E^{\frac{1}{2m(m+3)}}  \right\}.
   \end{align*}
   In view of triangle inequality, we obtain
   \begin{align*}
     \frac{d_{g(\delta_0)}(x_1, x_2)}{N^{-1}l}
     \leq \frac{ \sum_{i=0}^N d_{g(\delta_0)} \left(z_i, z_{i+1} \right)}{N^{-1}l}
     \leq  N\left\{1 + C\left\{ N^{-\frac{2}{m}} l^{\frac{2}{m}} + N^{\frac{1}{m}} l^{-\frac{1}{m}} E^{\frac{1}{2m(m+3)}}  \right\} \right\},
   \end{align*}
   which in turn implies that
   \begin{align}
     d_{g(\delta_0)}(x_1, x_2) l^{-1} \leq l + C\left\{ N^{-\frac{2}{m}} l^{\frac{2}{m}} + N^{\frac{1}{m}} l^{-\frac{1}{m}} E^{\frac{1}{2m(m+3)}}  \right\}.
	\label{eqn:f20_dup_N}
   \end{align}
   Let $N \sim lE^{-\frac{1}{6(m+3)}} > l\xi^{-1}$. Then (\ref{eqn:f20_dup_N}) yields that
   $\displaystyle d_{g(\delta_0)}(x_1, x_2) l^{-1} \leq 1 + CE^{\frac{1}{3m(m+3)}}$.
\end{proof}

Based on Theorem~\ref{thm:isometry_almost}, we are ready to prove a gap theorem.
\begin{theorem}[Gap theorem]
  There exists a big constant $L_0=L_0(m)$ such that the following properties hold.

  Suppose  $\left\{ (X, x_0, g(t)), 0 \leq t \leq 1 \right\}$ satisfies the same conditions as in Theorem~\ref{thm:pseudo-locality_A22}.
  Then for every $0<r<\frac{1}{4}$, we have
  \begin{align}
    r^{-1}d_{GH}\left(  \left( B_{g(0)}(x_0, r), g(0) \right), \left( B_{g(\delta_0)}(x_0, r), g(\delta_0) \right) \right)
    < L_0 r^{-1}E^{\frac{1}{3m(m+3)}},
    \label{eqnin:gdtwotimes}
  \end{align}
  whenever $\displaystyle E= \int_0^{2\delta_0} \int_{B_{g(0)}(x_0, 1)} |R-m\lambda_0| d\mu dt << r^{6(m+3)}$. Moreover, we have
  \begin{align}
    r^{-1}d_{GH}\left(  \left( B_{g(0)}(x_0, r), g(0) \right), \left( B(0, r), g_{\E} \right) \right)
    < L_0 r^2,
    \label{eqnin:gdeuc}
  \end{align}
  whenever $E<<r^{9m(m+3)}$, $r<<1$.  Here $B(0, r)$ is the ball with radius $r$ in the Euclidean space $\R^m$.
  \label{thm:Gap_A22}
\end{theorem}

\begin{proof}
By (\ref{eqn:f20_d2sides}), we have
\begin{align}
  \left| d_{g(0)}(x_1, x_2) - d_{g(\delta_0)}(x_1, x_2) \right| < C \max \left\{ E^{\frac{1}{3m(m+3)}}, E^{\frac{1}{2(m+3)}}\right\}<CE^{\frac{1}{3m(m+3)}}
  \label{eqn:sbdm}
\end{align}
for every two points $x_1, x_2 \in B_{g(0)} \left(x_0, \frac{1}{4} \right)$ satisfying $d_{g(0)}(x_1, x_2) >> E^{\frac{1}{6(m+3)}}$.    In particular, if
$d_{g(0)}(x_1, x_2)$ is comparable with $E^{\frac{1}{3m(m+3)}} >> E^{\frac{1}{6(m+3)}}$, then (\ref{eqn:sbdm}) holds.
This means that the identity map is a $CE^{\frac{1}{3m(m+3)}}$-approximation map from $\left( B_{g(0)}\left(x_0, r \right), g(0) \right)$ to
$\left( B_{g(0)}\left(x_0, r \right), g(\delta_0) \right)$.   Therefore, we have
\begin{align}
  d_{GH} \left( \left( B_{g(0)}\left(x_0, r \right), g(0) \right),  \left( B_{g(0)} \left(x_0, r \right),  g(\delta_0) \right) \right) < CE^{\frac{1}{3m(m+3)}}.
\label{eqn:f21_ghapp}
\end{align}
On the other hand, (\ref{eqn:f20_d2sides}) implies that
\begin{align*}
  B_{g(\delta_0)} \left(x_0, r-CE^{\frac{1}{2(m+3)}} \right)
  \subset B_{g(0)} \left(x_0, r \right) \subset B_{g(\delta_0)} \left(x_0, r+CE^{\frac{1}{3m(m+3)}} \right),
\end{align*}
which in turn yields that
\begin{align}
  d_{GH} \left( \left( B_{g(0)} \left(x_0, r \right), g(\delta_0) \right),  \left( B_{g(\delta_0)} \left(x_0, r \right),  g(\delta_0) \right)  \right)
  < CE^{\frac{1}{3m(m+3)}}
  \label{eqn:dbsm}
\end{align}
by the definition of Gromov-Hausdorff distance. Combine (\ref{eqn:f21_ghapp}) and (\ref{eqn:dbsm}), we obtain
\begin{align*}
   d_{GH} \left( \left( B_{g(0)} \left(x_0, r \right), g(0) \right),  \left( B_{g(\delta_0)} \left(x_0, r \right),  g(\delta_0) \right)  \right)
  < CE^{\frac{1}{3m(m+3)}},
\end{align*}
whose scaling-invariant form on the left hand side is (\ref{eqnin:gdtwotimes}).

At time $t=\delta_0$,  around $x_0$, $|Rm|$ is uniformly bounded, injectivity radius is uniformly bounded from below. Using exponential map,
one can construct approximation map from Euclidean ball to geodesic ball.  It is not hard to see that
\begin{align}
  r^{-1} d_{GH} \left( \left( B_{g(\delta_0)} \left(x_0, r \right),  g(\delta_0) \right),  \left( B(0, r), g_{\E} \right) \right) < Cr^2
\label{eqn:f20_deuc}
\end{align}
whenever $r$ is very small. It follows from (\ref{eqnin:gdtwotimes}) and (\ref{eqnin:gdeuc}) that
\begin{align*}
   r^{-1}d_{GH} \left( \left( B_{g(0)} \left(x_0, r \right), g(0) \right),  \left( B(0, r), g_{\E} \right)  \right)
   < C \left \{ r^2 + r^{-1}E^{\frac{1}{3m(m+3)}} \right\} < Cr^2
\end{align*}
whenever $E< r^{9m(m+3)}$.  Let $L_0$ be the maximum of all the  $C$'s that appear in this proof, we obtain Theorem~\ref{thm:Gap_A22}.
\end{proof}

\section{Structure of limit space}
This section is devoted to prove the structure theorems, Theorem~\ref{thmin:goodlimit} and Theorem~\ref{thmin:kgoodlimit}, respectively.

\subsection{Riemannian case}

Suppose $\left( X_i, x_i, g_i \right)$ is a sequence of almost Einstein manifolds.  Let $(\bar{X}, \bar{x}, \bar{g})$ be the limit space of $\left( X_i, x_i, g_i \right)$,
$\bar{\lambda}$ be the limit of $\lambda_i$.  In this section, we shall use the estimates developed in previous sections to show the structure of $\bar{X}$.

A tangent space $\left( \hat{Y}, \hat{y}, \hat{g} \right)$ at a point $y \in \bar{X}$ is the pointed-Gromov-Hausdorff limit of
$\left( \bar{X}, y, \epsilon_j^{-2} \bar{g} \right)$ for some sequence $\epsilon_j \to 0$.  A point $y \in \bar{X}$ is called regular if every tangent cone at $y$ is isometric to the Euclidean space $\left( \R^m, 0, g_{\E} \right)$.
A point $y \in X$ is called singular if it is not regular, i.e., at $y$, there exists a tangent space $\left( \hat{Y}, \hat{y}, \hat{g} \right)$ which is not
isometric to the Euclidean space.  By the fundamental work in~\cite{CC1}, one sees that every tangent space is a metric cone. Moreover, a tangent cone is Gromov-Hausdorff close to
the Euclidean space if and only if the volume of the standard unit ball in the tangent cone is close to $\omega_m$, the volume of the unit ball in $\R^m$.
Under the non-collapsed and Ricci lower bound condition, the Hausdorff measure converges whenever the Gromov-Hausdorff convergence happens.
This inspires us to define the function $\mathcal{U}$ on $\bar{X} \times (0,\infty)$ as follows.  For every point $y \in \bar{X}$, define
$\displaystyle \mathcal{U}(y,r) \triangleq \omega_m^{-1} r^{-m} |B(y,r)|$. Since the space $\bar{X}$ inherits the Bishop-Gromov volume comparison property from the limit process, we see that
$\displaystyle \lim_{r \to 0} \mathcal{U}(y,r)$ is a well defined positive number, which we denote by $\mathcal{U}(y)$.   Clearly, a point $y$ is singular if and only if
$\mathcal{U}(y)<1$.   However, by using the special property of almost Einstein limit, this property can be improved.
\begin{proposition}
  $y \in \bar{X}$ is a singular point if and only if  $\displaystyle \mathcal{U}(y) \leq \left( 1-\frac{\delta_0}{2} \right)$.
\label{prn:A6_1}
\end{proposition}
\begin{proof}
 It suffices to show that $y$ is regular whenever $\displaystyle \mathcal{U}(y)> \left( 1-\frac{\delta_0}{2} \right)$.

 Suppose $\displaystyle \mathcal{U}(y)> \left( 1-\frac{\delta_0}{2} \right)$. By definition of $\mathcal{U}(y)$, there exists a sequence of $\rho_j \to 0$ such that
 $$\omega_m^{-1}\rho_j^{-m}|B(y, \rho_j)| > \left(1- \frac{1}{2}\delta_0 \right).$$
 Denote the pointed-Gromov-Hausdorff limit of $\left( \bar{X}, y, \rho_j^{-2} \bar{g} \right)$ by $\left( \hat{Y}, \hat{y}, \hat{g}\right)$, which is a tangent cone of $\bar{X}$ at the point $y$. By a careful choice of diagonal subsequence if necessary, we can assume
  $\left( \hat{Y}, \hat{y}, \hat{g} \right)$ as the pointed-Gromov-Hausdorff limit of $\left( X_{i_j}, y_{i_j}, \rho_j^{-2} g_{i_j} \right)$, which is a new sequence of almost Einstein manifolds.  For brevity, we drop some subindexes and look  $\left( \hat{Y}, \hat{y}, \hat{g} \right)$ as the almost Einstein limit of $\left( X_{j}, y_{j}, h_{j} \right)$, where $h_j = \rho_j^{-2}g_{i_j}$.  By volume continuity, we have
  \begin{align*}
     \rho_j^{-m}|B(y, \rho_j)| > \omega_m \left(1- \frac{1}{2}\delta_0 \right), \Rightarrow
     |B(y_j, 1)|_{d\mu_{h_j}} > \left( 1-\delta_0 \right) \omega_m.
  \end{align*}
  Clearly, $Ric_{h_j} \geq -(m-1)\rho_j^{2}$ on $X_j$.  Therefore, Theorem~\ref{thm:Gap_A22} applies.  Fix an arbitrary small $r>0$, by inequality (\ref{eqnin:gdeuc}), we see that
  \begin{align*}
    r^{-1} d_{GH} \left( \left( B_{h_j}(y_j, r), h_j \right), \left( B(0,r), g_{\E} \right) \right)  <Cr^{2}
    \Rightarrow
    r^{-1} d_{GH} \left( \left( B_{\hat{g}}(\hat{y}, r), \hat{g}\right), \left( B(0,r), g_{\E} \right) \right)  \leq Cr^{2}.
  \end{align*}
  Consequently, every tangent space of $\hat{Y}$ at $\hat{y}$ is the Euclidean space $\R^m$. On the other hand, we already know $\hat{Y}$ is a metric cone
  with vertex $\hat{y}$.  These two conditions force that $\hat{Y}$ is isometric to $\R^m$. Henceforth, $y$ is a regular point.
\end{proof}

 By some routine argument, the following Corollary is obvious now.

\begin{corollary}
  There exists a constant $\bar{\epsilon}=\bar{\epsilon}(m)>0$ with the following property.

  Suppose $y \in \bar{X}$, $\left(\hat{Y}, \hat{y}, \hat{g}\right)$ is a tangent space of $\bar{X}$ at $y$,
  $B(0,1)$ is the unit ball in the Euclidean space $\R^m$.  Then $\hat{Y}$ is isometric to $\R^{m}$ if and only if
  \begin{align}
    d_{GH} \left( \left(B_{\hat{g}}(\hat{y}, 1), \hat{g} \right), \left( B(0,1), g_{\E} \right) \right)<\bar{\epsilon}.
    \label{eqn:A4_eclose}
  \end{align}
  \label{cly:A7_1}
\end{corollary}

 Using the notation of~\cite{CC1}, Corollary~\ref{cly:A7_1} implies $\mathcal{R}=\mathcal{R}_{\bar{\epsilon}}$. Therefore, we have separated the singular points from
 the regular points substantially.    Then by using regularity results from the Ricci flow, we can smoothen the regular part $\mathcal{R}$.

\begin{proposition}
  Suppose $y \in \bar{X}$ is a regular point. Then there exists a constant $r=r(y)$
  with the following properties.
  \begin{itemize}
    \item $\left( B_{\bar{g}}(y, r), \bar{g} \right)$ is geodesic convex, i.e., every shortest geodesic connecting two points in $B_{\bar{g}}(y, r)$ cannot escape it.
    \item There exist a region $D \subset \R^m$ and a smooth metric tensor $g_D$ on $D$ such that $\left( B_{\bar{g}}(y, r), \bar{g} \right)$ is isometric to
          $(D, g_D)$.
	\item $Ric_{\bar{g}}(y)-\bar{\lambda}\bar{g}(y)=0$.
  \end{itemize}
  \label{prn:A4_2}
\end{proposition}

\begin{proof}
  Since $y$ is regular, $\mathcal{U}(y)=1$. So we can find $r_0=r_0(y)$ such that $\mathcal{U}(y,\rho)>\left( 1-\frac{\delta_0}{2} \right)$ for every $0<\rho<r_0(y)$.
  Suppose $y_i \to y$ as $\left( X_i, x_i, g_i \right)$ converges to $\left( \bar{X}, \bar{x}, \bar{g} \right)$.  By volume continuity, we have for large $i$,
$$\displaystyle  r_0^{-m} \left| B_{g_i}(y_i, r_0)\right|_{d\mu_{g_i}} > \left( 1-\delta_0 \right) \omega_m.$$
Without loss of generality, we choose $r_0<\sqrt{\delta_0}$.
  Let $\tilde{g}_i=r_0^{-2}g_i$, $\Omega_i=B_{\tilde{g}_i}(y_i, 1)$.  Then we have
  \begin{align}
    Ric_{\tilde{g}_i}(x) \geq -(m-1)r_0^2>-(m-1)\delta_0, \; \forall \; x \in \Omega_i; \quad
    \left| B_{\tilde{g}_i}(y_i, 1)\right|_{d\mu_{\tilde{g}_i}} \geq \left( 1-\delta_0 \right) \omega_m.
  \label{eqn:A5_2}
  \end{align}
  So we can apply Theorem~\ref{thm:Gap_A22} for the new almost Einstein sequence $(X_i, y_i, \tilde{g}_i)$. By (\ref{eqnin:gdtwotimes}), it turns out that
  \begin{align}
    \lim_{i \to \infty} 8 d_{GH} \left( \left( B_{\tilde{g}_i(0)}\left(y_i, \frac{1}{8} \right), \tilde{g}_i(0) \right),
    \left( B_{\tilde{g}_i(\delta_0)} \left(y_i, \frac{1}{8} \right), \tilde{g}_i(\delta_0) \right)\right) =0.
    \label{eqn:A4_5}
  \end{align}
  Denote the common Gromov-Hausdorff limit of the two sequences of geodesic balls in (\ref{eqn:A4_5}) by
  $\left( B_{\tilde{g}_{\infty}}\left(y_{\infty}, \frac{1}{8} \right), \tilde{g}_{\infty} \right)$.
  Note that $B_{\tilde{g}_i(\delta_0)} \left(y_i, \frac{1}{8} \right) \subset B_{\tilde{g}_i(0)} \left(y_i, \frac{1}{2} \right)$ by Theorem~\ref{thm:isometry_almost}.
  Therefore, Theorem~\ref{thm:pseudo-locality_A22} and Shi's local estimate imply that there exist a small positive number  $\rho_0 << \min\left\{\frac{1}{8}, \delta_0 \right\}$
  and large positive constants $C_k$ such that
  \begin{align}
    \inf_{B_{\tilde{g}_i(\delta_0)}(y_i, \rho_0)} inj_{\tilde{g}_i(\delta_0)}(x) >> \rho_0;
    \quad \sup_{B_{\tilde{g}_i(\delta_0)}(y_i, \rho_0)} \left|\nabla ^k Rm \right|_{\tilde{g}_i(\delta_0)}(x) << C_k \rho_0^{-2-k}, \forall \; k \in \Z^{+} \cup \left\{ 0 \right\}.
    \label{eqn:A5_1}
  \end{align}
  Consequently, $\left( B_{\tilde{g}_{\infty}}\left(y_{\infty}, \rho_0 \right), \tilde{g}_{\infty} \right)$
  is a convex smooth geodesic ball.  Denote $h=r_0^2 \tilde{g}_{\infty}$. Of course, $\left( B_{h}\left(y_{\infty}, r_0\rho_0 \right), h\right)$
  is a convex smooth geodesic ball.  By exponential map with respect to $h$, we can find $D \subset \R^m$ and smooth $g_D$ such that $(D, g_D)$ is isometric
  to $\left( B_{h}\left(y_{\infty}, r_0\rho_0 \right), h\right)$, which is the Gromov-Hausdorff limit of  $(B(y_i, r_0\rho_0), g_i(0))$.  So we finish the proof of the first two
  properties by letting $r=r_0 \rho_0$.   The last property follows from Lemma~\ref{lma:Ricbyt}.  Actually, (\ref{eqn:A5_2}) guarantees that we can apply inequality (\ref{eqn:f20_ricup}) to obtain
  \begin{align*}
    \left|Ric_{g_i}-\lambda_i g_i \right|\left(y_i, \delta_0 r_0^2 \right)=
    r_0^{-2} \left|Ric_{\tilde{g}_i}-\lambda_ir_0^2 \tilde{g}_i \right|(y_i, \delta_0)
    < C \left\{ \int_0^{2\delta_0} \int_{X_i} \left|R-m r_0^2\lambda_i \right|_{\tilde{g}_i(t)} d\mu_{\tilde{g}_i(t)} dt \right\}^{\frac{1}{2}} \to 0,
  \end{align*}
  where $C=C(m,r_0,\delta_0)$. Since $\bar{g}$ is the smooth limit of $g_i(\delta_0r_0^2)$ around $y$, we obtain $Ric_{\bar{g}}(y)=\bar{\lambda} \bar{g}(y)$.
\end{proof}

  For brevity, for every point $x \in \bar{X}$, define the volume radius
  \begin{align}
    r_V(x) \triangleq \max \left\{ r>0 \left|   \mathcal{U}(x,r) \geq \left(1- \frac{1}{2}\delta_0 \right)  \right. \right\}
    \label{eqn:A6_rV}
  \end{align}
  whenever the set is nonempty. Otherwise, let $r_V(x)=0$.
  Define  $\displaystyle  \mathcal{V}_r \triangleq \left\{ x \in \bar{X} | r_{V}(x) \leq r \right\}$, the set of points whose volume radius is not greater than $r$.
  Clearly, $\mathcal{V}_{0}$ is nothing but the singular set $\mathcal{S}$.
  Using the notation in~\cite{CC1}, for a metric space $Z$, we assume $z^*$ is the vertex of the metric cone $C(Z)$. Then for every pair of small positive constants $\eta, \xi$ and
  radius $0<r<\xi$, we define
  \begin{align*}
  \mathcal{S}_{\eta, (r,\xi)}^{k} = \left\{ y \in \bar{X}  \left| \inf_{r<s<\xi}s^{-1} d_{GH} \left( B(y, s), B\left( \left( \underline{0}, z^* \right), s \right) \right) \geq \eta, \;
  \textrm{for all} \; \R^{k+1} \times C(Z) \right. \right\}.
 \end{align*}
 Note that our $\mathcal{S}_{\eta, \left( r, 1\right)}^{k}$ is $\mathcal{S}_{\eta, r}^{k}$ in~\cite{ChNa}.
 By Theorem 1.10 of~\cite{ChNa}, a standard rescaling argument shows that for every $\xi <1$ and $\eta <<1$,
$$\displaystyle \xi^{-m} \left| B(y, 2\xi) \cap \mathcal{S}_{\eta, (r,\xi)}^{m-2} \right| \leq C(m,\kappa, \eta) \left(\frac{r}{\xi}\right)^{2-\eta}$$
whenever $y \in B(\bar{x}, 2)$. Consequently, the non-collapsed condition and a ball-covering argument imply that
$$\displaystyle \left| B(\bar{x}, 2) \cap \mathcal{S}_{\eta, (r,\xi)}^{m-2} \right| \leq C(m,\kappa,\eta) \xi^{-2+\eta} r^{2-\eta}.$$
In particular, we have
$$\displaystyle \left|B(\bar{x}, 2) \cap \mathcal{S}_{\eta, (r,\eta)}^{m-2} \right| < C(m,\kappa,\eta) r^{2-\eta}.$$
Therefore, we can obtain
 \begin{align}
   \left| B(\bar{x}, 2) \cap \mathcal{V}_{r} \right| \leq C(m,\kappa,\eta) r^{2-\eta}
 \label{eqn:A6_8}
 \end{align}
 if we can prove $\mathcal{V}_{r} \subset \mathcal{S}_{\eta, (r,\eta)}^{m-2}$.  In fact, this relationship follows from the following Lemma.

\begin{lemma}
There exists a constant $\eta_0=\eta_0(m, \kappa)$ with the following property.

Suppose that $(Y, g)$ is an $m$-dimensional complete Riemannian manifold, $Ric(x) \geq  -(m-1)$ in a geodesic ball $B(y_0, 2)$, $|B(y_0, 1)| \geq \kappa$.
If $0<r<\eta<\eta_0$ and $\displaystyle r^{-m}|B(y_0,r)|=\left(1-\frac{\delta_0}{2}\right)\omega_m$, then for every metric space $Z$, we have
\begin{align}
  \inf_{r<s<\eta} s^{-1} d_{GH}\left( B(y_0,s), B\left( \left( \underline{0}, z^* \right), s \right) \right) \geq \eta,
  \label{eqn:A6_4}
\end{align}
where $z^*$ is the vertex of the metric cone $C(Z)$,  $\left( \underline{0}, z^* \right) \in \R^{m-1} \times C(Z)$.
\label{lma:A6_3}
\end{lemma}

\begin{proof}
  Otherwise, there exist a sequence of positive numbers $\eta_i \to 0$ and a sequence of Riemannian manifolds $(Y_i, y_i, h_i)$ with the given conditions violating the statements.
  \begin{itemize}
    \item  $r_i^{-m}|B(y_{i}, r_i)|_{d\mu_{h_i}} = (1-\delta_0)\omega_m$ for some $0<r_i<\eta_i$.
    \item  There exists $s_i \in (r_i, \eta_i)$ such that
      $s_i^{-1} d_{GH}\left( B(y_i, s_i),  B\left( \left( \underline{0}, z_i^* \right), s_i \right) \right) < \eta_i$ for some $\R^{m-1} \times C(Z_i)$.
  \end{itemize}
  Let $\tilde{h}_i=s_i^{-2} h_i$. Denote the pointed-Gromov-Hausdorff limit of $\left( B_{\tilde{h}_i}(y_i, 1), y_i, \tilde{h}_i \right)$ by $(\hat{B}, \hat{y}, \hat{g})$.
  By limit process, there exists a metric space $\hat{Z}$ such that
$$\hat{y}=\left(\underline{0}, \hat{z}^* \right) \in \R^{m-1} \times C(\hat{Z}),~~~\hat{B}=B\left( \left( \underline{0}, \hat{z}^* \right), 1\right).$$
Clearly, every tangent space of $\hat{y}$ is $\R^{m-1} \times C(\hat{Z})$, which must be $\R^{m}$ by~\cite{CC1}.
  Therefore, by the continuity of volume, we have
$$\displaystyle \lim_{i \to \infty} s_i^{-m} |B_{h_i}(y_i, s_i)|_{d\mu_{h_i}} = \lim_{i \to \infty}  |B_{\tilde{h}_i}(y_i, 1)|_{d\mu_{\tilde{h}_i}}=\omega_m,$$
which yields
   \begin{align*}
    \left(1- \frac{1}{2}\delta_0 \right) \omega_m = \lim_{i \to \infty} r_i^{-m} |B_{h_i}(y_i, r_i)|_{d\mu_{h_i}}
    \geq  \lim_{i \to \infty} s_i^{-m} |B_{h_i}(y_i, s_i)|_{d\mu_{h_i}}=\omega_m.
  \end{align*}
  by volume comparison. Contradiction!
\end{proof}

Suppose $r_V(y)=1$.  Let $y_i \to y$ as $(X_i, x_i, g_i)$ converges to $\left( \bar{X}, \bar{x}, \bar{g} \right)$.
Applying inequality (\ref{eqn:pseudo}) to the flow $\left\{ \left( X_i, y_i, g_i(t) \right), 0 \leq t \leq 1 \right\}$,
we have $\displaystyle |Rm|(y) = \lim_{i \to \infty} |Rm|_{g_i(\delta_0)}(y_i) \leq \delta_0^{-1}$.
By a trivial rescaling argument, we see that
\begin{align}
  |Rm|(y) \min \left\{r_{V}^2(y), 1 \right\}\leq \delta_0^{-1}
  \label{eqn:A6_7}
\end{align}
for every $y \in \mathcal{R}$. Follow the route of~\cite{ChNa} for the Einstein case, we can obtain some bounds of curvature integration on $\mathcal{R}$.
\begin{proposition}
  For every $0<p<1$ and $\rho \geq 1$, we have a constant $C=C(m,\kappa,p,\rho)$ such that
  \begin{align*}
    \int_{B(\bar{x}, \rho) \cap \mathcal{R}} |Rm|^{p} d\mu < C.
  \end{align*}
  \label{prn:A6_2}
\end{proposition}

\begin{proof}
Without loss of generality, we assume $\rho=1$. Fix $\eta < (1-p)$, we have
\begin{align*}
  \delta_0^{p} \int_{B(\bar{x}, 1)} |Rm|^{p} d\mu
  <\int_{B(\bar{x}, 1)} \min \left\{r_V^{-2p}, 1 \right\} d\mu
  <C \left( 1+  \frac{1}{1-2^{2(p-1)+\eta}}\right)<C(m,\kappa,p),
\end{align*}
where we used (\ref{eqn:A6_8}) and (\ref{eqn:A6_7}).
\end{proof}

\begin{proposition}
  $\dim_{\mathcal{H}} \mathcal{S} \leq m-2$.
  \label{prn:A7_1}
\end{proposition}
\begin{proof}
 It follows from inequality (\ref{eqn:A6_8}) and the fact that  $\dim_{\mathcal{H}} \mathcal{S}$ is an integer.
\end{proof}

Combine all the discussions in this subsection, we finish the proof of Theorem~\ref{thmin:goodlimit}.

\subsection{K\"ahler case}
Suppose $\left( M_i, x_i, g_i, J_i \right)$ is a sequence of almost K\"ahler Einstein manifolds.  Let $(\bar{M}, \bar{x}, \bar{g})$ be the limit space of
$\left( M_i, x_i, g_i \right)$, $\bar{\lambda}$ be the limit of $\lambda_i$, $\bar{M}=\mathcal{R} \cup \mathcal{S}$ be the regular-singular decomposition.

It is not hard to see that $\mathcal{R}$ has a complex structure
$\bar{J}$ compatible with $\bar{g}$ and $\nabla_{\bar{g}} \bar{J}=0$.  Actually, it suffices to prove the existence of such $\bar{J}$ locally.
Fix $y \in \mathcal{R}$. Let $r_0=\frac{1}{2} r_V(y)$. Suppose $y_i \to y$ as $\left( M_i, x_i, g_i \right)$ converges to $\left( \bar{M}, \bar{x}, \bar{g} \right)$.
By the construction of $\bar{g}$, we know that
$B_{\bar{g}}(y, r_0)$ is the smooth limit of $\left( B_{g_i(\delta_0 r_0^2)}(y_i, r_0),  g_i(\delta_0 r_0^2) \right)$. Therefore, the complex structure $J_i$ on
$B_{g_i(\delta_0r_0^2)}(y_i, \delta_0 r_0^2)$ converges to the limit complex structure $\bar{J}$, which is compatible with $\bar{g}$ and $\nabla_{\bar{g}} \bar{J}=0$.

For non-collapsed limit of K\"ahler manifolds with bounded Ricci curvature, it was shown that every non-Euclidean tangent cone can split at most $2n-4$ independent lines.
The argument was based on an $\epsilon$-regularity theorem(c.f.Theorem 5.2 of~\cite{Ch1}), which can be improved to obtain the following Lemma.

\begin{lemma}
There exists a constant $\xi_0=\xi_0(n, \kappa)$ with the following property.

Suppose $(N, y_0, h, J)$ is a complete K\"ahler manifold of complex dimension $n$, $Ric \geq  -(n-1)$ on $N$, $|B(y_0,1)| \geq \kappa$.
Suppose for the scales $0<r<\eta<\xi_0$, we have
\begin{itemize}
  \item $\displaystyle r^{-2n}|B(y_0,r)|=\left(1- \frac{\delta_0}{2} \right)\omega_{2n}$.
  \item $\displaystyle \sup_{r<s<\eta} s^{2-2n} \int_{B(y_0, 10s)} |Ric|d\mu<\eta$.
\end{itemize}
Then for every metric space $Z$, we have
\begin{align}
  \inf_{r<s<\eta} s^{-1} d_{GH}\left( B(y_0,s), B\left( \left( \underline{0}, z^* \right), s \right) \right) \geq \eta,
  \label{eqn:A7_4}
\end{align}
where $z^*$ is the vertex of the metric cone $C(Z)$,  $\left( \underline{0}, z^* \right) \in \R^{2n-3} \times C(Z)$.
\label{lma:A7_3}
\end{lemma}

\begin{proof}
  The proof follows the same route as that of Lemma~\ref{lma:A6_3}.

  If the statement was wrong, there exist a sequence of scales $(r_i, \eta_i)$ with $\eta_i \to 0$ and a sequence of K\"ahler manifolds $(N_i, y_i, h_i, J_i)$ with the given conditions violating the statements.
  \begin{itemize}
    \item  $r_i^{-2n}|B(y_{i}, r_i)|_{d\mu_{h_i}} = \left(1-\frac{\delta_0}{2}\right)\omega_{2n}$ for some $0<r_i<\eta_i$.
    \item  $\displaystyle \sup_{r_i <s <\eta_i} s^{-2n+2} \int_{B(y_i, 10s)} |Ric|_{h_i} d\mu_{h_i} <\eta_i$.
    \item  There exists $s_i \in \left(r_i, \eta_i \right)$ such that
      $s_i^{-1} d_{GH}\left( B(y_i, s_i),  B\left( \left( \underline{0}, z_i^* \right), s_i \right) \right) < \eta_i$ for some $\R^{2n-3} \times C(Z_i)$.
  \end{itemize}
  Let $\tilde{h}_i=s_i^{-2} h_i$. Denote the pointed-Gromov-Hausdorff limit of $\left( B_{\tilde{h}_i}(y_i, 1), y_i, \tilde{h}_i \right)$ by $(\hat{B}, \hat{y}, \hat{g})$.
  By limit process, there exists a metric space $\hat{Z}$ such that
  $$\hat{y}=\left(\underline{0}, \hat{z}^* \right) \in \R^{2n-3} \times C(\hat{Z}),~~~\hat{B}=B\left( \left( \underline{0}, \hat{z}^* \right), 1\right).$$
  Like the proof of Lemma~\ref{lma:A6_3}, in order to obtain a contradiction, it suffices to show that
  $\R^{2n-3} \times C(\hat{Z})$ is isometric to $\R^{2n}$.  Actually, the K\"ahler condition implies that $\R^{2n-3} \times C(\hat{Z})$ is either
  $\R^{2n}$ or $\R^{2n-2} \times C(S_t)$ for some circle with length $t \in (0, 2\pi)$. However, for metric $\tilde{h}_i$, we have
  \begin{align*}
    \int_{B_{\tilde{h}_i}(y_i, 10)} |Ric|_{\tilde{h}_i} d\mu_{\tilde{h}_i} = s_i^{-2n+2} \int_{B_{h_i}(y_i, 10s_i)} |Ric|_{h_i} d\mu_{h_i}
    \leq  \sup_{r_i<s<\eta_i} s^{-2n+2} \int_{B_{h_i}(y_i, 10s)} |Ric|_{h_i} d\mu_{h_i}<\eta_i \to 0.
  \end{align*}
  This is enough for us to choose good slice where the integration of $|Ric|$ is as small as possible(c.f. Theorem 5.2 of~\cite{Ch1}).
  Therefore,  Chern-Simons theory implies that $t=2\pi$. Consequently, $\R^{2n-3} \times C(\hat{Z})$ must be isometric to
  $\R^{2n}$ and we can obtain the desired contradiction!
\end{proof}

Fix the pair $(r, \eta)$ such that $0<r<\eta<\xi_0$.  Let $y$ be an arbitrary point in $B(\bar{x}, 2) \subset \bar{M}$, $y_i \in M_i$ such that
$y_i \to y$ as $\left( M_i, x_i, g_i \right)$ converges to $\left( \bar{M}, \bar{x}, \bar{g} \right)$.
Recall that $F_i=\int_{M_i} |Ric+\lambda_i g_i|_{g_i}d\mu_{g_i} \to 0$. For every $s \in (r,\eta)$, we have
\begin{align*}
   s^{2-2n} \int_{B_{g_i}(y_i, 10s)} |Ric|_{g_i} d\mu_{g_i} &\leq s^{2-2n} \int_{B_{g_i}(y_i, 10s)} \left\{ |Ric+\lambda_i g_i|_{g_i}+|\lambda_i| \sqrt{n}\right\}d\mu_{g_i}\\
     &\leq r^{2-2n}\int_{M_i} |Ric+\lambda_i g_i|_{g_i}d\mu_{g_i} + |\lambda_i| \sqrt{n} \left(s^{-2n}|B_{g_i}(y_i, 10s)|_{d\mu_{g_i}}\right) s^2\\
     &\leq r^{2-2n} F_i + 2\sqrt{n} \cdot \omega_{2n} \cdot 10^{2n} \cdot \eta^2.
\end{align*}
It follows that
\begin{align*}
   \sup_{r<s<\eta} s^{2-2n} \int_{B_{g_i}(y_i, 10s)} |Ric|_{g_i} d\mu_{g_i} \leq r^{2-2n} F_i + 2\sqrt{n} \cdot \omega_{2n} \cdot 10^{2n} \cdot \eta^2
     \leq 4\sqrt{n} \cdot \omega_{2n} \cdot 10^{2n} \cdot \eta^2 < \eta
\end{align*}
for large $i$, whenever $\eta$ is chosen very small. Therefore, Lemma~\ref{lma:A7_3} can be applied to obtain that
$\mathcal{V}_{r} \subset \mathcal{S}_{\eta, (r,\eta)}^{2n-4}$ on the limit space $\bar{M}$. Then we can apply Theorem 1.10 of~\cite{ChNa} to obtain that
\begin{align}
   \left| B(\bar{x}, 2) \cap \mathcal{V}_{r} \right| \leq C(n,\kappa,\eta) r^{4-\eta}.
 \label{eqn:A7_8}
\end{align}
From here, we can deduce the following two propositions without difficulty.
\begin{proposition}
  For every $0<p<2$ and $\rho \geq 1$, we have a constant $C=C(m,\kappa,p,\rho)$ such that
  \begin{align*}
    \int_{B(\bar{x}, \rho) \cap \mathcal{R}} |Rm|^{p} d\mu < C.
  \end{align*}
   \label{prn:A7_2}
\end{proposition}

\begin{proposition}
   $\dim_{\mathcal{H}} \mathcal{S} \leq m-4$.
  \label{prn:f22_kahlersingular}
\end{proposition}

Combine all the discussion in this section, we finish the proof of Theorem~\ref{thmin:kgoodlimit}.  Moreover,
Theorem~\ref{thmin:kgoodlimit} can be improved if we assume $\int_{M_i} |Rm|_{g_i}^{p}d\mu_{g_i}<C$ uniformly for some $2 \leq p \leq \frac{m}{2}$,
or we assume $n=p=2$. The proofs follow from the combination of the methods described in this section and that in~\cite{ChNa}.
Since the proofs do not contain new method and we do not know a substantial applications of such results, we omit the details here.

\section{Examples}

 In this section, we show two examples of almost K\"ahler Einstein sequences. The applications
 of the structure theorem (Theorem~\ref{thmin:kgoodlimit}) are also discussed.
 Actually, both examples come to our attention spontaneously when we try to study the geometric properties of K\"ahler manifolds.
 It is for this study that we develop the whole paper.

\subsection{Smooth minimal varieties of general type}

A smooth projective variety $M$ is called of general type if the Kodaira dimension of $M$ is equal to the complex dimension of $M$, i.e.,
$$\displaystyle \lim_{k \to \infty}  \frac{\log \dim H^0(K_M^k)}{\log k}=n.$$
It is called minimal if $K_M$ is numerically effective (nef), i.e.,
$K_M \cdot C \geq 0$ for every effective curve $C \subset M$.
Suppose $M$ is a smooth minimal variety, then it is easy to see that $M$ admits a K\"ahler Einstein
metric if and only if $K_M$ is ample,  by Yau's solution of Calabi conjecture.
Since there are a lot of smooth minimal varieties whose canonical classes are not ample,
we cannot expect to find a K\"ahler Einstein metric on each smooth minimal variety.
However, on each such variety, we can construct a sequence of almost K\"ahler Einstein metrics.

\begin{theorem}
  Suppose $M$ is a smooth minimal projective variety of general type, $J$ is the default complex structure.
  Then there is a point $x_0 \in M$ and a sequence of metrics $g_i$ with the following properties.
  \begin{itemize}
   \item $\displaystyle \lim_{i \to \infty} [\chi_i]=-2\pi c_1(M)$ where $\chi_i$ is the metric form compatible with both $g_i$ and $J$.
   \item $(M, x_0, g_i, J)$ is an almost K\"ahler Einstein sequence.
  \end{itemize}
  \label{thm:a2_1}
\end{theorem}

\begin{proof}
 There exists a nonnegative $(1,1)$-current $\chi$ with $[\chi]=-2\pi c_1(M)$.
 Fix an arbitrary metric form $\omega$ on $M$.  Then for every $\epsilon>0$, $[\chi+\epsilon \omega]$ is a positive class.
 By Yau's solution of Calabi conjecture (c.f.~\cite{Ca} and~\cite{Yau78}), we can find a metric form $\chi_{\epsilon}$ such that
 $Ric(\chi_{\epsilon}) + \chi_{\epsilon}= \epsilon \omega$.  Let $g_{\epsilon}$ be the metric tensor compatible with both $\chi_{\epsilon}$
 and $J$. Clearly, we have
 \begin{align}
    Ric(g_{\epsilon}) + g_{\epsilon} \geq 0.  \label{eqn:A16_4}
 \end{align}
 Then we run the normalized Ricci flow
 $$\displaystyle \D{}{t} g= -Ric - g$$
 from the initial metric $g_{\epsilon}$.
 Denote the metric form at time $t$ by $\chi_{\epsilon, t}$.
 Whenever $\chi_{\epsilon, t}$ is well defined, it satisfies
 \begin{align*}
   [\chi_{\epsilon,t}]= e^{-t}[\chi_{\epsilon}] + \left( 1-e^{-t} \right) [\chi]= [\chi] + \epsilon e^{-t}[\omega]>0.
 \end{align*}
 Therefore, for every $\epsilon>0$, the normalized Ricci flow initiating from $g_{\epsilon}$ exists forever(c.f.~\cite{Tsu} and~\cite{TZZ}).
 In view of (\ref{eqn:A16_1}), the condition $R+n \geq 0$ is preserved by the flow.  Therefore, we have
 \begin{align}
   &\quad \int_{0}^{1} \int_{M} |R+n| \chi_{\epsilon,t}^n dt \nonumber\\
   &=\int_{0}^{1} \int_{M} (R+n) \chi_{\epsilon,t}^n dt \nonumber\\
   &= n \int_{0}^{1} \left(  \int_{M} \left( \chi_{\epsilon,t} - \chi \right) \wedge \chi_{\epsilon,t}^{n-1} \right) dt \nonumber\\
   &= n\epsilon \int_{0}^{1} e^{-t}  \left( \int_M \omega \wedge \left( \chi + \epsilon e^{-t} \omega  \right)^{n-1} \right) dt \nonumber\\
   &< n\epsilon \int_{0}^{1} e^{-t}  \left( \int_M \omega \wedge \left( \chi + \omega  \right)^{n-1} \right) dt \nonumber\\
   &=nC \epsilon.  \label{eqn:A3_1}
 \end{align}
 At time $t=0$, we have $Ric(\chi_{\epsilon})+\chi_{\epsilon} \geq 0$, which implies
 \begin{align}
   \int_M |Ric+\chi_{\epsilon}| \chi_{\epsilon}^n \leq  \int_M \sqrt{n} (R+n) \chi_{\epsilon}^n
   = \epsilon \cdot n^{\frac{3}{2}} \int_M \omega \wedge \left( \chi + \epsilon \omega \right)^{n-1} < C(\chi, \omega) n^{\frac{3}{2}} \epsilon.
   \label{eqn:A3_2}
 \end{align}
 In view of the study of complex Monge-Ampere equation theory (c.f.~\cite{TZZ}), there exists an algebraically defined subvariety $\mathcal{B} \subset M$
 such that $\displaystyle \chi_{\epsilon} \stackrel{C^{\infty}}{\longrightarrow} \hat{\chi}, g_{\epsilon} \stackrel{C^{\infty}}{\longrightarrow} \hat{g}$
 on $M \backslash \mathcal{B}$, whenever $\epsilon \to 0$.
 Since $\hat{g}$ is a smooth metric on $M \backslash \mathcal{B}$, we can choose a small convex geodesic ball
 $B_{\hat{g}}(x_0, 2\xi_0) \subset M \backslash \mathcal{B}$. Let $\epsilon_i \to 0$, $g_i=g_{\epsilon_i}$. Then we have
 \begin{align}
    |B_{g_i}(x_0, 1)|_{d\mu_{g_i}} \geq |B_{g_i}(x_0, 2\xi_0)|_{d\mu_{g_i}} > |B_{\hat{g}}(x_0, \xi_0)| \triangleq \kappa
 \label{eqn:A16_3}
 \end{align}
 for large $i$. By definition, (\ref{eqn:A16_4}), (\ref{eqn:A16_3}), (\ref{eqn:A3_1}) and (\ref{eqn:A3_2}) together imply that $(M, x_0, g_i, J)$ is
 an almost K\"ahler Einstein sequence.
\end{proof}

In the proof of Theorem~\ref{thm:a2_1}, when $\omega$ and $x_0$ are fixed, the almost K\"ahler Einstein sequence depends on the choice of the sequence $\left\{\epsilon_i\right\}_{i=1}^{\infty}$.
It is natural to ask whether the limit space depends on the choice of the sequence $\left\{\epsilon_i\right\}_{i=1}^{\infty}$.
In fact, the answer is no.  In~\cite{TW1}, we proved that every limit space
$\left( \bar{M}, \bar{x}, \bar{g} \right)$ is the metric completion of $\left( M \backslash \mathcal{B}, x_0, \hat{g} \right)$, which is independent
of the choice of $\left\{\epsilon_i\right\}_{i=1}^{\infty}$.  Another interesting question is whether $\bar{M}$ has a variety structure.
Generally, we do not know the answer although this is expected.  However, when
$(M,J)$ satisfies the Chern number equality $\displaystyle \left\{c_1^2(M) - \frac{2(n+1)}{n}c_2(M) \right\} \cdot  c_1^{n-2}(M) = 0$,
then $\bar{M}$ does have a projective variety structure. Actually, in~\cite{TW1}, we will use Theorem~\ref{thmin:kgoodlimit}
to show that $\bar{M}$ is a global quotient of the complex hyperbolic space, henceforth it is a variety.

\subsection{Fano manifolds}
A complex manifold $(M, J)$ is called a Fano manifold if $-K_M$ is ample. By the Kodaira embedding theorem, such a manifold must be projective and admits a
K\"ahler structure. The existence of K\"ahler Einstein metrics on Fano manifolds is a folklore problem (c.f.~\cite{Tian97} and references therein).
In~\cite{Tian87}, the first author introduced the $\alpha$-invariant $\alpha(M)$ and proved that K\"ahler Einstein metrics exist whenever $\alpha(M)>\frac{n}{n+1}$.
If we only assume $\alpha(M) \geq \frac{n}{n+1}$, then the situation becomes subtle. It is not clear whether $\alpha(M) \geq \frac{n}{n+1}$ implies
the existence of K\"ahler Einstein metrics.
On the other hand,  the existence of K\"ahler Einstein metrics implies that Mabuchi's K-energy (c.f.~\cite{Ma} for definition) is bounded from below.
But there are examples(c.f.~\cite{Tian97},~\cite{ChenIV}) where the K-energy is bounded from below and K\"ahler Einstein metrics do not exist.
In short, neither $\alpha(M) \geq \frac{n}{n+1}$ nor the K-energy bounded from below can guarantee the existence of K\"ahler Einstein metrics.
However, either of them provides a sufficient condition for the existence of almost K\"ahler Einstein sequences.

 \begin{proposition}
  Suppose $(M, J)$ is a Fano manifold, $x_0 \in M$. Then in the class $2\pi c_1(M)$, there is a sequence of almost K\"ahler Einstein manifolds $(M, x_0, g_i, J)$
  if one of the following conditions are satisfied.
  \begin{itemize}
    \item  $\alpha(M) \geq \frac{n}{n+1}$.
   \item  Mabuchi's K-energy is bounded from below in $2\pi c_1(M)$.
  \end{itemize}
  \label{prn:A18_1}
\end{proposition}

Before we prove this proposition, let us recall an invariant.
Suppose $(M, J)$ is a Fano manifold, $\omega$ is a metric form in the class $2\pi c_1(M)$.
Since every other metric form in the same class can be written as $\omega_{\varphi}=\omega+\st \ddb \varphi$ for some smooth function $\varphi$ on $M$,
it is clear that
\begin{align*}
  \sup \left\{ t>0 \left| Ric(\omega_{\varphi}) \geq t \omega_{\varphi} \; \textrm{for some}\; \varphi \in C^{\infty}(M)  \right. \right\}
\end{align*}
is independent of the choice of $\omega$.  For brevity, we denote this invariant by $\mathcal{G}(M,J)$, or by $\mathcal{G}(M)$ when no ambiguity happens.
Under this notation, we have the following theorem.

\begin{theorem}
  Suppose $(M, J)$ is a Fano manifold with $\mathcal{G}(M)=1$, $x_0 \in M$.
  Then there is a sequence of metrics $g_i$ with the following properties.
 \begin{itemize}
   \item $[\omega_i] \in 2\pi c_1(M)$ where $\omega_i$ is the metric form compatible with both $g_i$ and $J$.
   \item $(M,x_0,g_i,J)$ is an almost K\"ahler Einstein sequence.
 \end{itemize}
 \label{thm:h9_1}
\end{theorem}

\begin{proof}
Since $\mathcal{G}(M)=1$, for every $0<\alpha<1$, there is a metric form $\omega_{\alpha}$ with $ Ric(\omega_{\alpha}) \geq \alpha \omega_{\alpha}$.
Let $g_{\alpha}$ be the metric tensor compatible with both $\omega_{\alpha}$ and $J$. Clearly, we have
\begin{align}
   Ric(g_{\alpha}) \geq \alpha g_{\alpha}.  \label{eqn:A16_5}
\end{align}
Let $\alpha_i \to 1$, $\omega_i=\omega_{\alpha_i}$, $g_i=g_{\alpha_i}$. Then we have
\begin{align}
  \int_{M} |Ric_{g_i} -g_i| \omega_i^n &\leq \int_{M} \left\{ |Ric_{g_i}-\alpha_i g_i| + (1-\alpha_i) g_i \right\} \omega_i^n \nonumber\\
   &\leq \sqrt{n} \int_{M} \left\{R-n\alpha_i + n(1-\alpha_i) \right\} \omega_i^n \nonumber\\
   &=2n^{\frac{3}{2}} (1-\alpha_i) \cdot (2\pi)^n c_1^n(M) \to 0.
   \label{eqn:h9_ric0}
\end{align}
Initiating from $g_i$, we run the normalized Ricci flow
$$\D{}{t} g\,=\,-Ric + g,$$
which preserves the cohomology class $2\pi c_1(M)$.
Since $R-n \geq n(\alpha_i-1)$ at the initial time, it follows from (\ref{eqn:A16_1}) that
 \begin{align*}
   \left( R-n \right)_{g_i(t)} \geq -n(1-\alpha_i)e^t \Rightarrow  R_{g_i(t)} \geq n\left\{ 1-(1-\alpha_i)e^t \right\}.
 \end{align*}
Consequently, we have
 \begin{align}
   \int_0^1 \int_M |R-n|_{g_i(t)} \omega_i^n(t)dt&=\int_0^1 \int_M \left| R-n\left\{ 1-(1-\alpha_i)e^t \right\} - n(1-\alpha_i)e^t \right|_{g_i(t)} \omega_i^n(t)dt \nonumber\\
   &\leq \int_0^1 \int_M \left\{ R-n\left\{ 1-(1-\alpha_i)e^t \right\} + n(1-\alpha_i)e^t \right\} \omega_i^n(t)dt \nonumber\\
   &=2n(1-\alpha_i) \cdot (2\pi)^n c_1^n(M) \cdot  \int_0^1 e^t dt \nonumber\\
   &=2n(e-1) \cdot (2\pi)^n c_1^n(M) \cdot (1-\alpha_i) \to 0. \label{eqn:h9_r0}
 \end{align}
 Since $\alpha_i \to 1$, we can assume $\alpha_i > \frac{2n-1}{2n}$.  So Bonett-Myers theorem implies a diameter upper bound
 $\diam_{g_i} M< \sqrt{2n} \pi$. By Bishop volume comparison, we have
 \begin{align}
   \frac{|B(x_0, 1)|_{d\mu_{g_i}}}{|M|_{d\mu_{g_i}}}=\frac{|B(x_0, 1)|_{d\mu_{g_i}}}{\left|B \left(x_0, \sqrt{2n} \pi \right) \right|_{d\mu_{g_i}}} \geq C(n)
   \Rightarrow  |B(x_0, 1)|_{d\mu_{g_i}} \geq C(n, c_1^n(M)) \triangleq \kappa,
   \label{eqn:A16_2}
 \end{align}
 which is the non-collapsed condition.  Therefore, by definition, (\ref{eqn:A16_5}), (\ref{eqn:A16_2}), (\ref{eqn:h9_r0}) and (\ref{eqn:h9_ric0})
 yields that $(M_i,x_0,g_i,J)$ form a sequence of almost K\"ahler Einstein  manifolds.
 \end{proof}

Note that $\mathcal{G}(M) =1$ under either condition of Proposition~\ref{prn:A18_1} (c.f.~\cite{Sze}).
Therefore, Proposition~\ref{prn:A18_1} follows from Theorem~\ref{thm:h9_1}. \\

 In both examples, Theorem~\ref{thm:a2_1} and Theorem~\ref{thm:h9_1},
 the complex structure is fixed. This is of course not needed in the set up of almost K\"ahler Einstein manifolds.
 Therefore, potentially, we should be able to construct almost K\"ahler Einstein sequences by deforming the complex structure and
 cohomology class simultaneously.   It is then interesting to see whether the almost K\"ahler Einstein limit space
 is independent of the choice of parameter (of complex structures and metric forms) sequences.  It is also fascinating to ask
 whether the limit space has a variety structure.  These topics will be studied in the future.

Gang Tian, BICMR and SMS, Beijing University and Department of Mathematics, Princeton University, tian@math.princeton.edu\\

Bing  Wang, Simons Center for Geometry and Physics, State University of New York at Stony Brook, bwang@scgp.stonybrook.edu\\
\end{document}